\documentclass[10pt,twoside,a4paper,abstract=on]{scrartcl}

\usepackage{authblk}

\usepackage{geometry}
\usepackage{multicol}

\usepackage[utf8]{inputenc}
\usepackage{hyperref}
\usepackage{caption}

\usepackage{todonotes}

\usepackage[toc,title,page]{appendix}

\usepackage{amsmath,amsfonts,amssymb,amsthm}
\usepackage{bbm}
\usepackage{stackengine} % for precise stacking

\usepackage{enumitem}

\usepackage{tikz}  

\usepackage{algorithm}
\usepackage{algpseudocode}

\newtheorem{Theorem}{Theorem}[section]

\newtheorem{Definition}[Theorem]{Definition}
\newtheorem{Proposition}[Theorem]{Proposition}
\newtheorem{Lemma}[Theorem]{Lemma}
\newtheorem{Assumption}[]{Assumption}

\newtheorem{Remark}[Theorem]{Remark}

\usepackage[a-1b]{pdfx}

\usepackage{graphicx}
\usepackage{subcaption}

\usepackage[maxbibnames=50]{biblatex}
\definecolor{electricultramarine}{rgb}{0.25, 0.0, 1.0}
\definecolor{ikb}{rgb}{0.0, 0.18, 0.65}
\definecolor{green(colorwheel)(x11green)}{rgb}{0.16, 0.5, 0.0}

\usepackage[normalem]{ulem}
\newcommand*{\fzcst}[1]{\relax\ifmmode\text{\textcolor{green(colorwheel)(x11green)}{\sout{\ensuremath{#1}}}}\else\textcolor{green(colorwheel)(x11green)}{\sout{#1}}\fi}

\usepackage{todonotes,varwidth}

\newcommand*{\ykcst}[1]{\relax\ifmmode\text{\textcolor{ikb}{\sout{\ensuremath{#1}}}}\else\textcolor{ikb}{\sout{#1}}
\fi}

\title{\vspace{-0.2cm}Existence 
	of Dynamical Low-Rank \\Approximation 
	for SDEs \\ with Locally Lipschitz Coefficients}
\author[1]{Yoshihito Kazashi}
\author[2]{Fabio Nobile}
\author[2]{Fabio Zoccolan}
\affil[1]{Department of Mathematics,
	The University of Manchester, M13 9PL, UK. email: y.kazashi@manchester.ac.uk}
\affil[2]{Institut de Math\'ematiques, \'Ecole Polytechnique F\'ed\'erale de Lausanne, 1015 Lausanne, Switzerland. email: fabio.nobile@epfl.ch, fabio.zoccolan@epfl.ch}
\date{}
\begin{document}
   
   \maketitle
	
	\begin{abstract}
		\noindent 
		In this work, we consider the Dynamical Low Rank Approximation (DLRA) of a high-dimensional system of Stochastic Differential Equations (SDEs); DLRA is an effective reduced order model that projects the equations on a time-varying low dimensional subspace following the evolution of the solution.
		
		Well posedness results for the resulting DLRA equations have been proved in \cite{kazashi2025dynamical} under the assumption of global Lipschitz continuity and linear growth of the drift and the diffusion coefficients. In this work, we extend the existence result to the case of locally Lipschitz coefficient and we prove uniqueness of solutions under the assumption of a mild growth of the Lipschitz constant to infinity.
			\end{abstract}	
    %\tableofcontents
    
    %\listoftodos

	\section*{Introduction}
	\addcontentsline{toc}{section}{Introduction}
	Simulations of physical and engineering systems have been a cornerstone of modern science and technology, providing predictive insights when experiments cannot be afforded; for instance, being for instance too expensive to pursue or presenting a high level of risk.
	Beyond these causes, a proper mathematical bottleneck has to be faced. To accurately find multi-scale solutions in fields such as computational fluid dynamics, structural mechanics, and climate modeling, one usually runs into high-dimensional simulations.
	Such simulations often require the discretization of partial differential equations (PDEs) in very large state spaces, leading to prohibitive computational costs, especially in real-time or many-query contexts. The situation can be even harder when dealing with randomness, as the \textit{curse of dimensionality} can arise from the stochastic space, too. 
	
	A possibility to overcome this issue is provided by ROM techniques. These methods usually approximate full-order, high-fidelity models by projecting the governing equations onto low-dimensional subspaces. This strategy allows to drastically reduce the computational costs to efficiently approximate the full-order dynamics when the original problem presents a low-rank structure.
	ROM strategies incorporate techniques such as Proper Orthogonal Decomposition (POD), Balanced Truncation, and Reduced Basis methods, which have been extensively applied in the field of deterministic or random PDEs or numerical linear-algebra \cite{benner2005dimension,benner2020model,hesthaven2016certified,quarteroni2015reduced,schilders2008model}.
	
	For several physical and real-life fields, such as finance, weather forecasting, or biology, the employment of SDEs is absolutely beneficial to model the phenomena studied. Also in this context, the need for efficient high-dimensional simulations is increasingly on high demand. However, reduced order modeling entails several nontrivial issues, including the preservation of stochastic features, stability, and accuracy across varying regime which remains an open area of research.
	
	Among the various model reduction approaches proposed, the so-called \textit{Dynamical Low-Rank Approximation} (DLRA) has started to gain relevance in the context of SDEs, showing promising results in numerical performances and applications.
	DLRA was first proposed for matrix ODEs in \cite{koch2007dynamical} and then it was extensively applied in other contexts with outstanding results, too, for instance, random and deterministic PDEs (see e.g.\ \cite{ceruti2022unconventional,einkemmer2019quasi,einkemmer2021asymptotic,koch2007dynamical,lubich2014projector,kieri2016discretized,kazashi2021existence,kazashi2021stability,bachmayr2021existence,carrel2025interpolatory, sapsis2009dynamically}).
	A first well-posedness analysis of DLRA for SDE was proposed in \cite{kazashi2025dynamical} under standard conditions of Lipschitzianity and linear-growth bound for the coefficients of the SDEs, while a study on its long-time behavior can be found in \cite{bao2026exponential}. A numerical analysis of its time and stochastic discretization was presented in \cite{kazashi2026dynamicalpartI,kazashi2026dynamicalpartII}. Regarding applications, a first study on DLRA applied for Kalman-Bucy filtering SDEs has been proposed in \cite{nobile2025dynamical}. We also cite \cite{cao2018stochastic}, where DLRA-type equations for SDEs are derived from a finite signed measure approach.
	
	In \cite{kazashi2025dynamical}, DLRA is derived using the so-called \textit{Dynamically-Orthogonal} formalism, first introduced in \cite{sapsis2009dynamically}. This methodology seeks to construct an approximation consisting of two distinct components: a deterministic basis and a stochastic counterpart, both of which are allowed to evolve in time and are computed adaptively.  Unlike traditional ROM techniques, such reduced basis \cite{quarteroni2015reduced}, in DLRA both deterministic and stochastic components are allowed to evolve in time. This feature permits to better track the solution in problems characterized by low dimensional features, which evolve considerably over time, while maintaining a low computational cost. The reference \cite{kazashi2025dynamical} will be the starting point of this work.
	
	\emph{The aim of this work is to extend the well posed analysis of DLRA to a larger class of SDEs.}
	
	More specifically, existence of DLRA for SDEs is studied under the condition of local Lipschizianity and linear-growth bound of the drift and diffusion, where the proof is achieved via an application of Schauder's theorem. This approach is different from that in \cite{kazashi2025dynamical}, where a Picard-iteration strategy was developed, and therefore offers a new interesting perspective. This treatment can be found in Section \ref{sec: well-posedness ll}. Furthermore, in Section \ref{sec: ll uniqueness}, we prove a first uniqueness result considering a mild growth of the Lipschitz constant to infinity.
	
	\section{Problem setting}
	%In this section, we propose an analysis of existence of DLRA of SDEs with local Lipschitz coefficients satisfying linear-growth condition. The setting is the following.
	In this section, we present the general SDE setting which our analysis relies upon.
	Let us consider a stochastic basis $\left( \Omega, \mathcal{F}, \mathbb{P}, (\mathcal{F}_t)_{t \geq 0} \right)$, where $\Omega$ is the probability domain, $\mathcal{F}$ is a $\sigma$-algebra on $\Omega$, $\mathbb{P}$ is a probability measure on $\Omega$, and $(\mathcal{F}_t)_{t \geq 0}$ is a filtration on the probability space $\left(\Omega, \mathcal{F}, \mathbb{P}\right)$, satisfying the standard assumptions \cite[Remark 6.24]{schilling2021brownian}.
	We consider a real $m$-dimensional $(\mathcal{F}_t)$-Brownian motion $W(t) = \left(W_1(t), \ldots , W_m(t) \right)^{\top}$
	and the following SDE
	\begin{equation}\label{eq:SDE-int}
		X^{\mathrm{true}}(t)= X^{\mathrm{true}}_0 + \int_{0}^{t}a(s,X^{\mathrm{true}}(s))\mathrm{d}t+\int_{0}^{t}b(s,X^{\mathrm{true}}(s))\mathrm{d}W_s, \quad \forall t \in [0, +\infty),
	\end{equation}
	where the solution of \eqref{eq:SDE-int} is a vector valued stochastic process $X^{\mathrm{true}}(t)=\left(X^{\mathrm{true}}_1(t), \ldots, X^{\mathrm{true}}_d(t)\right)^{\top}$, $d \in \mathbb{N}$, whereas the drift $a\colon[0,\infty)\times\mathbb{R}^{d}\to\mathbb{R}^{d}$ 
	and the diffusion $b\colon[0,\infty)\times\mathbb{R}^{d}\to\mathbb{R}^{d\times m}$
	are measurable between the Borel fields $\left([0,\infty)\times\mathbb{R}^{d};\mathcal{B}([0,\infty)\times\mathbb{R}^{d})\right)$ and $\left(\mathbb{R}^{d};\mathcal{B}(\mathbb{R}^{d})\right)$, $\left(\mathbb{R}^{d\times m};\mathcal{B}(\mathbb{R}^{d \times m})\right)$ respectively, $\mathcal{B}$ denoting the Borel $\sigma$-algebra.
	
	The DLR approximation of \eqref{eq:SDE-int} is a $d$-dimensional process $X_t$ that factorizes at each time $t$ as $X_t =U_t^{\top}Y_t=\sum_{i = 1}^kU_t^{i}Y_t^{i}$ with deterministic $U_t \in \mathbb{R}^{k \times d}$ and stochastic $Y_t \in L^2(\Omega, \mathbb{R}^k)$ where the pair $\left(U,Y\right)$ satisfies the so-called DO equations~\cite{kazashi2025dynamical,sapsis2009dynamically}:
	\begin{align}
		{C}_{Y_{t}}\dot{U}_{t} & =\mathbb{E}[Y_{t}a(t,U_{t}^{\top}Y_{t})^{\top}](I_{d\times d}-P_{U_{t}}^{\mathrm{row}}),\label{eq:DLR-eq-U}\\
		\mathrm{d}Y_{t} & =U_{t}a(t,U_{t}^{\top}Y_{t})\,\mathrm{d}t+U_{t}b(t,U_{t}^{\top}Y_{t})\mathrm{d}W_{t},\label{eq:DLR-eq-Y}
	\end{align}
	with $C_{Y_t}:= \mathbb{E}[Y_tY_{t}^{\top}]$ the Gram matrix (or Gramian) of the stochastic basis $Y_t$, and $P^{\text{row}}_{U_t}$ the projector matrix onto the vector space $\operatorname{span}\{U^{1}_t, \ldots, U^{k}_t\} \subset \mathbb{R}^{d}$, where $U^{i}_t$ is the $i$-th row of  $U_t$. If $U_t$ has orthonormal rows, then $P^{\text{row}}_{U_t}=U_t^{\top}U_t$ and the Gramian linked to $U_t$ is $U_t U_t^{\top} = I_{k \times k}$. Notice that \eqref{eq:DLR-eq-U} and \eqref{eq:DLR-eq-Y} are strongly coupled equations, hence their well-posedness is not trivial. Moreover, the evolution of the deterministic modes in \eqref{eq:DLR-eq-U} depends on the law of the process making \eqref{eq:DLR-eq-U}--\eqref{eq:DLR-eq-Y} of McKean-Vlasov type.
	
	Given a suitable rank-$k$ approximation $X_0 :=U_0^{\top}Y_0$ of the initial condition $X_0^{\text{true}}$ of \eqref{eq:SDE-int}, whose rank is assumed to be at least $k$ and which preserves the moment bound of $X_0^{\text{true}}$, in this work we consider the following definition of a \textit{strong DO solution} \cite{kazashi2025dynamical}.
	
	\begin{Definition}[strong DO solution of rank $k$]\label{def: D0 sol}
		A function $(U,Y) : [0,T] \to \mathbb{R}^{k\times d} \times L^{2}(\Omega;\mathbb{R}^{k})$ is called a \textit{strong DO solution} of rank $k$ for \eqref{eq:SDE-int} if the following conditions are satisfied:
		\begin{enumerate}
			\item {the initial condition $(U_0,Y_0 )$ is such that $U_0  \in \mathbb{R}^{k\times d}$ is a matrix with orthonormal rows and $Y_0  \in L^{2}(\Omega;\mathbb{R}^{k})$ has linearly independent components};
			\item the curve $t \to U_t \in \mathbb{R}^{k\times d}$ is absolutely continuous on $[0,T]$ and $U_t\dot{U}^{\top}_t = 0 \in \mathbb{R}^{k \times k}$ for a.e. $t \in [0,T]$;
			\item the curve $t \to Y_t(\omega) \in  \mathbb{R}^{k}$ has almost surely continuous paths on $[0,T]$ and is $\mathcal{F}_{t}$-measurable for all $t \in [0,T]$. Moreover,  for any $t \in [0,T]$ the components $Y^{1}_t,\dots,Y^{k}_t$ are linearly independent in $L^2(\Omega)$;
			\item $U$ satisfies equation \eqref{eq:DLR-eq-U} for a.e.~$t \in [0,T]$ and $Y$ is a strong solution of \eqref{eq:DLR-eq-Y} 
			on $[0,T]$.
		\end{enumerate}
	\end{Definition}
	In \cite{kazashi2025dynamical}, it is shown that a unique strong DO solution $(U,Y)$ exists at least for small times under Lipschitzianity and linear-growth bound conditions on the drift $a$ and the diffusion $b$. The same conditions also guarantee well-posedness of the original problem \eqref{eq:SDE-int} (see e.g.\ \cite[Theorem 3.1]{mao2007stochastic} or \cite[Theorem 2.9]{karatzas2012brownian}). 
	
	\section{Existence under Local Lipschitzianity}\label{sec: well-posedness ll}
	In this section, we extend the existence analysis of \cite{kazashi2025dynamical} beyond the global Lipschitz setting. Henceforth, we work under the following assumptions.
	\begin{Assumption}[Moment boundedness of the initial condition]\label{ass: initial value 2+eps}
		\begin{equation}\label{eq:initial value 2+eps}
			X^{\mathrm{true}}_0 \mbox{ is } \mathcal{F}_0\mbox{-measurable and satisfies } \mathbb{E}[|X^{\mathrm{true}}_0|^{2+\varepsilon}] < +\infty, \text{ for some } \varepsilon>0.
		\end{equation}
	\end{Assumption}
	\begin{Assumption}[Local Lipschitz]\label{ass: local-lipschitzianity}
		For any $n \in \mathbb{N}$, there exists a constant $L_n >0$ such that for all $|x|,|y| \leq n$ and for all $t\geq0$ it holds
		\begin{equation}\label{local-lipschitzianity}
			| a(t,x) - a(t,y) | + \| b(t,x) - b(t,y) \|_{\mathrm{F}} \leq L_n | x - y |,
		\end{equation}
		where $\|\cdot\|_{\mathrm{F}}$ denotes the Frobenius norm of a matrix.
	\end{Assumption}
	\begin{Assumption}[Linear-growth bound]\label{ass: linear-growth bound}
		The drift $a$
		and the diffusion $b$ fulfill the following linear-growth bound condition:
		\begin{equation}
			|a(t,x)|^{2}+\|b(t,x)\|_{\mathrm{F}}^{2}\leq C_{\mathrm{lgb}}(1+|x|^{2}), \quad \forall x \in \mathbb{R}^d, t \geq 0, \label{eq:lin-growth}
		\end{equation}
		for some constant $C_{\mathrm{lgb}}>0$.
	\end{Assumption}
	Under Assumptions \ref{ass: initial value 2+eps},\ref{ass: local-lipschitzianity}, and \ref{ass: linear-growth bound}, equation \eqref{eq:SDE-int} has a unique strong solution (see for example \cite[Theorem 3.4]{mao2007stochastic}).
	
    To prove existence of a strong DO solution, we consider the following strategy that is streamlined here first, before going in a more detailed treatment. To begin, we aim to find reasonable closed ``neighborhoods" $\mathbb{D}_{\mathrm{det}}^{\mathrm{loc}}$ and $\mathbb{D}_{\mathrm{sto}}^{\mathrm{loc}}$ of the initial data $U_0$ and $Y_0$, thought as constant functions in time on a suitable interval $[0,T_{\mathrm{loc}}]$, where each element belonging to these sets has linearly independent components in $[0,T_{\mathrm{loc}}]$. 
	This construction implies that the Gram matrix associated to any element in $\mathbb{D}_{\mathrm{sto}}^{\mathrm{loc}}$ and $\mathbb{D}_{\mathrm{det}}^{\mathrm{loc}}$ is nondegenerate in $[0,T_{\mathrm{loc}}]$. We convert then the problem \eqref{eq:DLR-eq-U}-\eqref{eq:DLR-eq-Y} into a fixed point problem $U=F(U)$ with a well defined map $F$ from a subset of $\mathbb{D}_{\mathrm{det}}^{\mathrm{loc}}$ into itself and apply Schauder's theorem to prove existence. We detail these steps in the next subsection. We give here a result on a priori norm bounds on the stochastic and deterministic bases, which will be useful throughout.
    
    \begin{Lemma}[Bound on the DLRA]\label{lem: DLR Gronwall - lg}
    	Suppose that Assumptions \ref{ass: initial value 2+eps} and \ref{ass: linear-growth bound} hold and that a solution $(U,Y)$ of equations \eqref{eq:DLR-eq-U}--\eqref{eq:DLR-eq-Y} exists in $[0,T]$, with $T >0$.
    	Then, there exists a positive constant $C_{T,\varepsilon}>0$ such that we have for all $t \in [0,T]$:
    	\begin{equation}\label{lg-stability and U norm}
    		\| U_t\|_{\mathrm{F}} = \sqrt{k} \ \forall t\in [0,T] \quad \mbox{and} \quad \mathbb{E}\left[\sup_{t\in[0,T]}|Y_{t}|^{2+\varepsilon}\right]
    		\leq C_{T,\varepsilon} \left( 1+\mathbb{E}[|Y_0|^{2+\varepsilon}]\right).
    	\end{equation}
    	\begin{proof}
    		By assumption, $U_t$ is absolutely continuous on $[0,T]$ and thus a.e.\ differentiable. As $U$ solves \eqref{eq:DLR-eq-U}, we have
    		\begin{equation*}
    			\dot{U}_t U_t^{\top}  = 0, %C^{-1}_{Y_t} \mathbb{E}\left[Y_t a(t,U_{t}^{\top}Y_t)^{\top}\right]\left(I_{d \times d} - P^{\text{row} }_{U_t} \right)U^{\top}_t =  C^{-1}_{Y_t} \mathbb{E}\left[Y_t a(t,U_{t}^{\top}Y_t)^{\top}\right](U^{\top}_t - U^{\top}_t) = 0,
    		\end{equation*}
    		for a.e.\ $t \in [0,T]$, as $\left(I_{d \times d}-P_{U_t}^{\mathrm{row}}\right)U_t = 0$, which implies $U_t U_t^{\top} =I_{k \times k}$ for all $t \in [0,T]$ since $U_0U_0^{\top}= I_{k \times k}$.
    		Hence $\|U_t\|_{\mathrm{F}} =\sqrt{ \operatorname{trace} (U_tU_t^{\top})} = \sqrt{k}$. 
    		
    		Let us consider now the stochastic basis and define the stopping time
    		\begin{equation*}
    			\tau_n(\omega):=\inf\left\{ t\geq 0 :\ |Y_t(\omega)|\geq n\right\}\wedge T, \quad \omega \in \Omega,
    		\end{equation*}
    		for $n \in \mathbb{N}$. For the sake of notation, denote $p:=2+\varepsilon$.
    		Then, for $u \in [0,T]$ one has that 
    		\begin{equation*}
    			\begin{aligned}
    				\mathbb{E}[ \sup_{t\in[0,u]} |Y_{t\wedge\tau_n}|^{p}]  & = \mathbb{E}\left[ \sup_{t\in[0,u]} |Y_{0} + \int_{0}^{t\wedge\tau_n} U_s a(s,X_s) \mathrm{d}s + \int_{0}^{t\wedge\tau_n} U_s b(s,X_s) \mathrm{d}W_s  |^{p}\right] \\
    				\leq & 3^{p-1} \mathbb{E}[|Y_{0}  |^{p}] + 3^{p-1} \mathbb{E}\left[ \sup_{t\in[0,u]} \left| \int_{0}^{t\wedge\tau_n} U_s a(s,X_s) \mathrm{d}s \right|^{p}\right] \\
    				&+ 3^{p-1} \mathbb{E}\left[ \sup_{t\in[0,u]} \left| \int_{0}^{t\wedge\tau_n} U_s b(s,X_s) \mathrm{d}W_s  \right|^{p}\right] \\
    				\leq & 3^{p-1} \mathbb{E}[|Y_{0}  |^{p}] + 3^{p-1} T^{p-1} \mathbb{E}\left[ \sup_{t\in[0,u]} \int_{0}^{t} \left| \mathbbm{1}_{\{s < \tau_n\}} U_{s} a(s,X_{s}) \right|^{p} \mathrm{d}s \right] \\
    				&+ 3^{p-1} \mathbb{E}\left[ \sup_{t\in[0,u]} \left| \int_{0}^{t} \mathbbm{1}_{\{s < \tau_n\}} U_{s} b(s,X_{s})  \mathrm{d}W_s  \right|^{p}\right] \\
    				\leq &3^{p-1} \mathbb{E}[|Y_{0}  |^{p}] + 3^{p-1} T^{p-1} \mathbb{E}\left[ \sup_{t\in[0,u]} \int_{0}^{t} \left| \mathbbm{1}_{\{s < \tau_n\}} a(s,X_{s}) \right|^{p} \mathrm{d}s \right] \\
    				&+ 3^{p-1} \mathbb{E}\left[ \sup_{t\in[0,u]} \left| \int_{0}^{t} \mathbbm{1}_{\{s < \tau_n\}} U_{s} b(s,X_{s})  \mathrm{d}W_s  \right|^{p}\right] \\
    				\leq & 3^{p-1} \mathbb{E}[|Y_{0}  |^{p}] + 3^{p-1}2^{\frac{p}{2}-1} T^{p-1} \mathbb{E}\left[  \int_{0}^{u} C_{\mathrm{lgb}}^{\frac{p}{2}} \left( 1 + \left|X_{s \wedge \tau_n}\right|^{p} \right) \mathrm{d}s \right] \\
    				&+ 3^{p-1} \mathbb{E}\left[ \sup_{t\in[0,u]} \left| \int_{0}^{t} \mathbbm{1}_{\{s < \tau_n\}} U_{s} b(s,X_{s})  \mathrm{d}W_s  \right|^{p}\right] \\
    			\end{aligned}
    		\end{equation*}
    		Then, via the norm bound on $U$, the Burkholder-Davis-Gundy inequality \cite[Theorem 3.28]{karatzas2012brownian}, and Hölder inequality, there exists a positive constant $C_p$ such that
    		\begin{equation*}
    			\begin{aligned}
    				\mathbb{E}\left[ \sup_{t\in[0,u]} \left| \int_{0}^{t} \mathbbm{1}_{\{s < \tau_n\}} U_{s} b(s,X_{s}) \mathrm{d}W_s  \right|^{p}\right] \leq &  C_p  \mathbb{E}\left[ \left( \int_{0}^{u} \left|  \mathbbm{1}_{\{s < \tau_n\}} U_{s} b(s,X_{s}) \right|^2 \mathrm{d}s \right)^{\frac{p}{2}}  \right] \\
    				\leq &  C_p  T^{\frac{p}{2}-1}\mathbb{E}\left[ \int_{0}^{u} \left|\mathbbm{1}_{\{s < \tau_n\}} b(s,X_{s}) \right|^{p} \mathrm{d}s \right] \\
    				\leq &  C_p   2^{\frac{p}{2}-1} T^{\frac{p}{2}-1} \mathbb{E}\left[ \int_{0}^{u} C_{\mathrm{lgb}}^{\frac{p}{2}} \left( 1 + \left|X_{s \wedge \tau_n}\right|^{p} \right) \mathrm{d}s \right] \\
    				\leq & C_p  2^{\frac{p}{2}-1} T^{\frac{p}{2}-1} \int_{0}^{u} C_{\mathrm{lgb}}^{\frac{p}{2}} \left( 1 + \mathbb{E}\left[ \left|X_{s \wedge \tau_n}\right|^{p}\right] \right) \mathrm{d}s,  \\
    			\end{aligned}
    		\end{equation*}
    		hence, we get
    		\begin{equation*}
    			\begin{aligned}
    				\mathbb{E}[ \sup_{t\in[0,u]} |Y_{t\wedge\tau_n}|^{p}]  
    				& \leq 3^{p-1} \mathbb{E}[|Y_{0}  |^{p}] + 3^{p-1}2^{\frac{p}{2}-1} \left( T^{p-1}  + C_p T^{\frac{p}{2}-1}\right) \int_{0}^{u} C_{\mathrm{lgb}}^{\frac{p}{2}} \left( 1 + \mathbb{E}\left[\left|X_{s \wedge \tau_n}\right|^{p} \right) \right] \mathrm{d}s \\
    				& \leq 3^{p-1} \mathbb{E}[|Y_{0}  |^{p}] + 3^{p-1}2^{\frac{p}{2}-1} \left( T^{p-1}  + C_pT^{\frac{p}{2}-1} \right) C_{\mathrm{lgb}}^{\frac{p}{2}} \int_{0}^{u} \left( 1 + \mathbb{E}\left[ \sup_{r\in[0,s]} \left|X_{r \wedge \tau_n}\right|^{p} \right] \right) \mathrm{d}s.
    			\end{aligned}
    		\end{equation*}
    		By Grönwall's lemma, it holds then
    		\begin{equation*}
    			\begin{aligned}
    				\mathbb{E}\left[\sup_{t\in[0,u]}|Y_{t\wedge\tau_n}|^{2+\varepsilon}\right]
    				&\leq C_{T,\varepsilon, u} \left( 1+\mathbb{E}[|Y_0|^{2+\varepsilon}]\right),
    			\end{aligned}
    		\end{equation*}
    		and, hence, for $u=T$, it follows
    		\begin{equation*}
    			\begin{aligned}
    				\mathbb{E}\left[\sup_{t\in[0,T]}|Y_{t\wedge\tau_n}|^{2+\varepsilon}\right]
    				&\leq C_{T,\varepsilon} \left( 1+\mathbb{E}[|Y_0|^{2+\varepsilon}]\right),
    			\end{aligned}
    		\end{equation*}
    		where \(C_{T,\varepsilon, u}>0\) and \(C_{T,\varepsilon}>0\) do not depend on \(n\). As  \(C_{T,\varepsilon}>0\) is independent of $n$, letting \(n\to\infty\) and using Fatou's lemma yield the thesis
    		\begin{equation*}
    			\begin{aligned}
    				\mathbb{E}\left[\sup_{t\in[0,T]}|Y_t|^{2+\varepsilon}\right]&\leq C_{T,\varepsilon} \left(1+\mathbb{E}[|Y_0|^{2+\varepsilon}]\right).
    			\end{aligned}
    		\end{equation*}
    	\end{proof}
    \end{Lemma}
    Similarly to the proof of Lemma \ref{lem: DLR Gronwall - lg}, one can prove that there exists a positive constant $K(T)$ such
    \begin{equation}
    	\sup\limits_{t \in [0, T]}\mathbb{E}[|Y_t|^2] \leq \big(\mathbb{E}[|Y_0|^2] + C_{\mathrm{lgb}}T \big)e^{(C_{\mathrm{lgb}}+1)T}=:K(T).
    \end{equation}

    \subsection{Construction of the neighborhoods $\mathbb{D}_{\mathrm{det}}^{\mathrm{loc}}$ and $\mathbb{D}_{\mathrm{sto}}^{\mathrm{loc}}$ and fixed point map $F$}
    To define $\mathbb{D}_{\mathrm{det}}^{\mathrm{loc}}$ and $\mathbb{D}_{\mathrm{sto}}^{\mathrm{loc}}$, we take advantage of \cite[Proposition A.1]{kazashi2025dynamical}, whose statement is recalled here for the sake of clarity.
    
    \begin{Proposition} \label{prop: A inverse}
    	For a Hilbert space $(H,\langle\cdot,\cdot\rangle)$,
    	denote by $[H]^{k}$ the product Hilbert space equipped with the
    	norm $\|Z\|_{[H]^{k}}=\sqrt{\sum_{j=1}^{k}\|Z^{j}\|_{H}^{2}}$
    	for $Z=(Z^{j})_j\in[H]^{k}$. For $Z\in[H]^{k}$
    	let $C_{Z}$ be the Gram matrix 
    	\[
    	C_{Z}:=\left(\langle{Z}^{j},{Z}^{\ell}\rangle\right)_{j,\ell=1,\dots k}\in\mathbb{R}^{k\times k}.
    	\]
    	
    	Suppose that $Z_{0}\in[H]^{k}$ has linearly independent
    	components $Z^{j}_0$, $j=1,\dots,k$ in $H$, and that $\|Z_{0}\|_{[H]^k}\leq\rho$
    	and $\|C_{Z_{0}}^{-1}\|_{\mathrm{F}}\leq\gamma$ hold
    	for $\rho,\gamma>0$. Then, there exists $\eta:=\eta(\rho,\gamma)>0$ {defined as
    		\begin{equation}
    			\eta(\rho,\gamma):=-\rho+\sqrt{\rho^{2}+\frac{1}{2\gamma}}\label{eq:def-eta}
    		\end{equation}
    	}such that we have 
    	\[
    	\|C_{Z}^{-1}\|_{\mathrm{F}}\leq2\gamma,\quad\text{for any }Z\in B_{\eta}(Z_{0}),
    	\]
    	and $\eta(\rho,\gamma)$ is decreasing in both $\rho$ and $\gamma$.
    	Here, $B_{\eta}(Z_{0})$ is the open ball in $[H]^{k}$
    	of radius $\eta$ around $Z_{0}$.
    \end{Proposition}
    
   Defining $\rho:=\|Y_0\|_{[L^{2}(\Omega)]^{k}}$ and $\gamma :=\|{C}_{Y_0}^{-1}\|_{\mathrm{F}}$, and observing that $\|U_0\|_{\mathrm{F}} = \|(U_0U_0^{\top})^{-1}\|_{\mathrm{F}}= \sqrt{k}$, with slight abuse of notation, we redefine
   \begin{equation}\label{eq: eta}
   	\eta=\eta(k,\rho,\gamma):=\min\{-\sqrt{k}+\sqrt{k+\frac{1}{2\sqrt{k}}},-\rho+\sqrt{\rho^{2}+\frac{1}{2\gamma}}\},
   \end{equation}
   where Proposition \ref{prop: A inverse} tells us that $\eta$ is non-increasing in $\rho$, $\gamma$, and $k$. 
   
   In \cite{kazashi2025dynamical}, we have proved existence of strong DLRA solutions for globally Lipschitz coefficients on the time interval $[0, T]$ with $T=T(\eta)$ defined as
   \begin{equation}\label{eq:T}
   	T(\eta):=\min\{1,\frac{
   		\min\{1,
   		\eta(k,\rho,\gamma)^2 
   		\}}
   	{36kC_{\mathrm{lgb}}(1+3k(3\rho^2 +1))},
   	\frac{
   		\min\{
   		\eta(k,\rho,\gamma)^2
   		,k\}
   	}{8\gamma^2(3\rho^2+1) C_{\mathrm{lgb}}(1+3k(3\rho^2 +1))}\},
   \end{equation}
   using a Picard-type fixed point argument on the \emph{``non-closed"} neighborhoods
   \[
   \mathbb{D}_{\mathrm{det}}:=\left\{ V \in C([0,\delta];\mathbb{R}^{k\times d})\,\left\vert \,\begin{array}{l}
   	\sup\limits_{0\leq t\leq\delta}\|V_{t}\|_{\mathrm{F}}^{2}\leq3k\text{, and}\\
   	V_{t}\in B_{\eta}(U_0)\text{ for }t\in[0,\delta]
   \end{array}\right\} \right.
   \]
   and
   \[
   \mathbb{D}_{\mathrm{sto}}:=\left\{ Z \in L^{2}(\Omega;C([0,\delta];\mathbb{R}^{k}))\,\left\vert \,\begin{array}{l}
   	\ensuremath{Z}\text{ is \ensuremath{\mathcal{F}_{t}}-adapted, }\\
   	\mathbb{E}\bigl[\sup\limits_{0\leq t\leq\delta}|Z_{t}|^{2}\bigr]\leq 3\rho^2+1\text{, and}\\
   	Z_{t}\in B_{\eta}(Y_0)\text{ for }t\in[0,\delta]
   \end{array}\right\} \right. .
   \]
   
   Under the more general Assumption \ref{ass: local-lipschitzianity} of local Lipschitz coefficients considered in this work, extending the Picard argument is not straightforward and
   we use instead Schauder theorem, which requires however closed neighborhoods. We therefore consider a (possibly shorter) final time $T_{\mathrm{loc}} \leq T(\eta)$, a larger bound on the stochastic basis, and smaller closed balls
   contained in $B_{\eta}(U_0)$ and $B_{\eta}(Y_0)$, for instance by halving the radius $\eta \mapsto \frac{\eta}{\sqrt{2}}$. We thus define 
   \begin{equation}\label{eq:T loc}
   	\begin{aligned}
   	T_{\mathrm{loc}}:=&\min\Big\{1,\frac{	\min\{1,
   		\frac{\eta(k,\rho,\gamma)^2 }{2}
   		\}}
   	{90k^2 C_{\mathrm{lgb}} \left(3 \rho^2 +2\right)e^{3k}}, \frac{
   		\min\{
   		\frac{\eta(k,\rho,\gamma)^2 }{2}
   		,k\}
   	}{4\gamma C_{\mathrm{lgb}} \left(3 \rho^2 +2 \right) 3k^2 e^{3k} }\Big\},
      	\end{aligned}
   \end{equation}
   and the closed neighborhoods of the constant functions $U_0$ and $Y_0$ as
   \begin{equation}\label{eq:D_det loc}
   	\mathbb{D}_{\mathrm{det}}^{\mathrm{loc}}:=\left\{ V \in C([0,T_{\mathrm{loc}}];\mathbb{R}^{k\times d})\,\left\vert \,\begin{array}{l}
   		\sup\limits_{0\leq t\leq T_{\mathrm{loc}}}\|V_{t}\|_{\mathrm{F}}^{2}\leq3k\text{, and}\\
   		V_{t}\in \overline{B_{\frac{\eta}{\sqrt{2}}}(U_0)}\text{ for }t\in[0,T_{\mathrm{loc}}]
   	\end{array}\right\} \right.
   \end{equation}
   and
   \begin{equation}\label{eq:D_sto loc}
   	\mathbb{D}_{\mathrm{sto}}^{\mathrm{loc}}:=\left\{ Z \in L^{2}(\Omega;C([0,T_{\mathrm{loc}}];\mathbb{R}^{k}))\,\left\vert \,\begin{array}{l}
   		\ensuremath{Z}\text{ is \ensuremath{\mathcal{F}_{t}}-adapted, }\\
   		\mathbb{E}\bigl[\sup\limits_{0\leq t\leq T_{\mathrm{loc}}}|Z_{t}|^{2}\bigr]\leq \left( 3 \rho^2 + 1 \right)e^{3k} \text{, and}\\
   		Z_{t}\in \overline{B_{\frac{\eta}{\sqrt{2}}}(Y_0)}\text{ for }t\in[0,T_{\mathrm{loc}}]
   	\end{array}\right\} \right. .
   \end{equation}
   Proposition \ref{prop: A inverse} still implies that all elements $(V_t)_t \in \mathbb{D}_{\mathrm{det}}^{\mathrm{loc}}$ and $(Z_t)_t \in \mathbb{D}_{\mathrm{sto}}^{\mathrm{loc}}$ have linearly independent components in $\mathbb{R}^{d}$ and in $L^2(\Omega)$, respectively, for all $t$, hence their associated Gram matrices are invertible on $[0, T_{\mathrm{loc}}]$.
\emph{For the sake of notation, from now on we will denote $T_{\mathrm{loc}}$ in \eqref{eq:T loc} by $T$.}
  
 Let us define the set
 	\begin{equation}\label{eq: C}
 		C := \{U \in \mathbb{D}_{\mathrm{det}}^{\mathrm{loc}}, \ U(0) = U_0 \}.
 		\end{equation}
   and then
   the map 
   \begin{equation}
   	\begin{aligned}
   		F : & \ C \longrightarrow C \\
   		& \  U \longmapsto F(U) := F_2(U, F_1(U)) \label{eq: F} 
   	\end{aligned}
   \end{equation} 
with
   	\begin{align}
   		U  \rightarrow F_1(U)=:& \tilde{Y}  \text{ such that }   \mathrm{d}\tilde{Y}_{t} =U_{t}a(t,U_{t}^{\top}\tilde{Y}_{t})\,\mathrm{d}t+U_{t}b(t,U_{t}^{\top}\tilde{Y}_{t})\mathrm{d}W_{t},\ \tilde{Y}(0) = Y_0, \label{eq: F_1} \\
         (U,\tilde{Y})  \rightarrow F_2(U, \tilde{Y}) =: &\tilde{U}  \text{ such that }  \frac{\mathrm{d}{\tilde{U}}_{t}}{\mathrm{d}t} ={C}_{\tilde{Y}_{t}}^{-1}\mathbb{E}[\tilde{Y}_{t}a(t,U_{t}^{\top}\tilde{Y}_{t})^{\top}](I_{d\times d}-P_{\tilde{U}_{t}}^{\mathrm{row}}), \ \tilde{U}(0) = U_0.\label{eq: F_2} 
   	\end{align}

  Notice that $\tilde{U}$ remains orthogonal at all times in $[0,T]$, thus, we can write $P_{\tilde{U}_t}^{\mathrm{row}} = \tilde{U}^{\top}_t\tilde{U}_t$. If $F$ admits a fixed point $U \in\mathbb{D}_{\mathrm{det}}^{\mathrm{loc}}$ and $Y=F_1(U)$, then $(U,Y)$ is a strong solution of \eqref{eq:DLR-eq-U} and \eqref{eq:DLR-eq-Y} in the time interval $[0,T]$. 
   The content of the next section is centered on proving the existence of the fixed point of $F$.

\begin{Remark}[Independence of \eqref{eq:T loc} from Lipschitz assumption]

	A central difficulty in proving the existence of a DO solution is the presence of the inverse Gramian ${C}_{\tilde{Y}_{t}}^{-1}$, which makes the analysis highly non-standard. 
	The quantity \eqref{eq:T loc} provides a time interval in which we can control the invertibility of ${C}_{\tilde{Y}_{t}}$.
 	It is important to observe that expression \eqref{eq:T loc} depends on the initial data $U_0$ and $Y_0$, as well as on the linear-growth bound of the drift and diffusion. However, it does not depend on their Lipschtzianity property. This insight implies that the same result still holds independently of Assumption \ref{ass: local-lipschitzianity} and opens the door to further extensions of the theory to a weakened Assumption \ref{ass: local-lipschitzianity}.
 \end{Remark}

    \subsection{Existence of a fixed point}
    
    This section focuses on studying the local existence of solutions of \eqref{eq:DLR-eq-U}-\eqref{eq:DLR-eq-Y}. The main result is contained in Theorem \ref{thm: well-posedness locl-linear-growth}, stated at the end of this section. Its proof relies on Schauder Theorem \cite[Theorem 8.8]{deimling2013nonlinear}, which says that, given a nonempty, closed, bounded, and convex subset $C$
	of a real Banach space $\mathfrak{C}$, a map $F: C \to C$ that is compact, i.e.\ a continuous map whose images of subsets of $C$ are relatively compact, admits a fixed point, namely $F(x)=x$ for $x \in C$. Our definitions of $C \subseteq \mathfrak{C}:=C( [0,T];\mathbb{R}^{k \times d})$ in \eqref{eq: C} and $F$ in \eqref{eq: F} with $T$ defined as in \eqref{eq:T loc} are suitable to achieve this goal. We recall that $U_0$ is chosen with orthonormal rows. Before stating the main theorem and its proof we present several Lemmata and Propositions proving the necessary properties of $C$ and $F$. We start with the properties of $C$.
    
    \begin{Lemma}[Properties of the set $C$]\label{lem: C}
    	The set $C$ defined in \eqref{eq: C}
    	is a subset of $\mathfrak{C}=C( [0,T];\mathbb{R}^{k \times d})$, and is nonempty, closed, bounded, and convex.
    	\begin{proof}
       Trivially, $C \subset \mathfrak{C}$, and the function $(U_t)_t$ such that $U_t = U_0$ for all $t\geq 0$ belongs to $C$ by definition, implying that $C$ is nonempty. The subset $C$ is also closed due to the $sup$ norm, and bounded, as by construction one has that $\sup_{t} \|U_t\|_{\mathrm{F}} \leq \sqrt{3k}$ for all $U \in C$. Furthermore, if $U, V \in C$, for any $\lambda \in [0,1]$ one has that
    	$\lambda U_t + (1-\lambda) V_t$ is continuous in $[0,T]$ and
    	$$\|\lambda U_t + (1-\lambda) V_t\|_{\mathrm{F}} \leq |\lambda| \|U_t\|_{\mathrm{F}} + |1-\lambda| \|V_t\|_{\mathrm{F}} \leq |\lambda| \sqrt{3k} + |1-\lambda| \sqrt{3k} = \sqrt{3k}.$$ 
	 Similarly, $\overline{B_{\frac{\eta}{\sqrt{2}}}(U_0)}$ is closed under convex combination. Therefore, $C$ is convex.
    	    \end{proof}
    \end{Lemma}
    Notice that by choosing $C$ as in \eqref{eq: C} we are not asking for the orthogonality of the elements $U \in C$ as in this case $C$ would not be convex. We will retrieve orthogonality of the function $U$ after having proved the existence of a fixed point of $F$. In order to prove that $F$ is well-defined, the following lemma is fundamental.
    
    \begin{Lemma}\label{lem: tilde U and Y in neigh}
    	Suppose Assumptions \ref{ass: initial value 2+eps} and \ref{ass: linear-growth bound} to hold. If $U \in \mathbb{D}_{\mathrm{det}}^{\mathrm{loc}}$, then $\tilde{Y} = F_1(U)$ satisfies $\tilde{Y} \in \mathbb{D}_{\mathrm{sto}}^{\mathrm{loc}}$ if it exists. 
    		Furthermore, there exists a positive constant $\tilde{K}_{p}(T)$ such that
    		\begin{equation}\label{eq: tilde K p + eps}
    			\begin{aligned}
    				\mathbb{E}[ \sup_{t\in[0,T]} |\tilde{Y}_{t}|^{p}] \leq \tilde{K}_{p}(T)
    			\end{aligned}
    		\end{equation}
    		with 
    		\begin{equation*}
    			\begin{aligned}
    			\tilde{K}_{p}(T):=\left(3^{p-1} \mathbb{E}[|Y_{0}  |^{p}] + 3^{p-1}2^{\frac{p}{2}-1} (3k)^{\frac{p}{2}} \left( T^{p}  + C_pT^{\frac{p}{2}} \right) C_{\mathrm{lgb}}^{\frac{p}{2}} \right) \\
    			\qquad \cdot \ \mathrm{exp}\left(3^{p-1}2^{\frac{p}{2}-1} (3k)^{p} C_{\mathrm{lgb}}^{\frac{p}{2}} \left( T^{p}  + C_pT^{\frac{p}{2}} \right) \right).
    			\end{aligned}
    		\end{equation*}
    	
    	Similarly, if $\tilde{Y} \in \mathbb{D}_{\mathrm{sto}}^{\mathrm{loc}}$, then $\tilde{U} = F_2(U,\tilde{Y})$ satisfies $\tilde{U} \in \mathbb{D}_{\mathrm{det}}^{\mathrm{loc}}$ if it exists.
    	\begin{proof}
    		We first start from the $p$-moment-bound on $\tilde{Y}$. One can proceed with similar computations to Lemma \ref{lem: DLR Gronwall - lg}, but without assuming $U$ to be orthogonal. Indeed, for the stopping time $\tau_n(\omega):=\inf\left\{ t\geq 0 :\ |\tilde{Y}_t(\omega)|\geq n\right\}\wedge T,$ for $\omega \in \Omega$ and for $p=2 +\varepsilon$ one has
    			\begin{equation}\label{eq: interm K_tilde}
    				\begin{aligned}
    					\mathbb{E}[ \sup_{t\in[0,u]} |\tilde{Y}_{t\wedge\tau_n}|^{p}]  
    					& \leq 3^{p-1} \mathbb{E}[|Y_{0}  |^{p}] + 3^{p-1}2^{\frac{p}{2}-1} (3k)^{\frac{p}{2}}\left( T^{p-1}  + C_p T^{\frac{p}{2}-1}\right) \\
    					& \qquad \qquad \cdot \int_{0}^{u} C_{\mathrm{lgb}}^{\frac{p}{2}} \left( 1 + \mathbb{E}\left[\left|U_{s \wedge \tau_n}^{\top}\tilde{Y}_{s \wedge \tau_n}\right|^{p} \right) \right] \mathrm{d}s \\
    					& \leq 3^{p-1} \mathbb{E}[|Y_{0}  |^{p}] + 3^{p-1}2^{\frac{p}{2}-1} (3k)^{\frac{p}{2}} \left( T^{p-1}  + C_pT^{\frac{p}{2}-1} \right)\\
    					 & \qquad \qquad \cdot C_{\mathrm{lgb}}^{\frac{p}{2}} \int_{0}^{u} \left( 1 + (3k)^{\frac{p}{2}} \mathbb{E}\left[ \sup_{r\in[0,s]} \left|\tilde{Y}_{r \wedge \tau_n}\right|^{p} \right] \right) \mathrm{d}s,
    				\end{aligned}
    			\end{equation}
    			where the presence of the $3^{p-1}2^{\frac{p}{2}-1}$ is an artifact of the bound $|a(t,x)|^{p} \leq (C_{\mathrm{lgb}} (1 + |x|^2)^{\frac{p}{2}} \leq 2^{\frac{p}{2}-1} C_{\mathrm{lgb}}^{\frac{p}{2}}  (1 + |x|^p)$.
    			Then, the Gronwall's lemma and the Fatou's lemma imply relation \eqref{eq: tilde K p + eps}.
    		Notice that for $p=2$ one has $C_p=4$, and \eqref{eq: tilde K p + eps} becomes 
    			\begin{equation*}
    				\mathbb{E}[ \sup_{t\in[0,u]} |\tilde{Y}_{t}|^{2}]   \leq \left( 3 \mathbb{E}[|Y_{0}  |^{2}] + 3 (3k) \left( T  + 4 \right) C_{\mathrm{lgb}} T \right)  \mathrm{exp} \left( 3 (3k)^2 \left( T  + 4 \right) C_{\mathrm{lgb}} T \right).
    			\end{equation*}
    			By choosing $u=T(=T_{\mathrm{loc}} \leq 1)$ one obtains that 
    			\begin{equation}\label{eq: rho bound Y tilde}
    				\begin{aligned}
    					\mathbb{E}[ \sup_{t\in[0,T]} |\tilde{Y}_{t}|^{2}]  \leq & \left( 3 \rho^2 + 45k C_{\mathrm{lgb}} T \right)  \mathrm{exp} \left( 3k 45 k C_{\mathrm{lgb}} T \right) \\ 
    					\leq  & (3 \rho^2 + 1) e^{3k} ,\\
    				\end{aligned}
    			\end{equation}
    		where we use the fact that $T \leq \frac{1}{45k C_{\mathrm{lgb}}}$.
    			Furthermore, with similar computations one can write that
    			\begin{equation*}
    				\begin{aligned}
    					\mathbb{E}[ \sup_{t\in[0,T]} |\tilde{Y}_{t} - Y_0|^{2}] \leq & 2\mathbb{E}\biggl[\sup_{t\in[0,T]} t\int_{0}^{t} \|U_{s}\|_{\mathrm{F}}^2 |a(s,U_s^{\top}\tilde{Y}_{s})|^{2}\,\mathrm{d}s+\sup_{t\in[0,T]} \Bigl|\int_{0}^{t}U_s b(s,U_s^{\top}\tilde{Y}_{s})\mathrm{d}W_{s}\Bigr|^{2}\biggr]\\
    					\leq & 2(3k)^2 (T+4)  C_{\mathrm{lgb}} \left(1 + \mathbb{E}[ \sup_{t\in[0,T]} |\tilde{Y}_{t}|^{2}] \right) T \\
    					\leq & 90k^2 C_{\mathrm{lgb}} \left(1 + \left(3 \rho^2 + 1 \right)  e^{3k} \right) T \leq   90k^2 C_{\mathrm{lgb}} \left(3 \rho^2 +2 \right)  e^{3k}  T\\
    					\leq &  \frac{\eta^2}{2},\\
    				\end{aligned}
    			\end{equation*}
    			where we have employed the bound \eqref{eq: rho bound Y tilde} as well as \eqref{eq:T loc}.
    		
    		Concerning the claim on the deterministic basis $\tilde{U}$, by similar computations to the one of \cite[Lemma 3.1]{kazashi2025dynamical}, one has
    			\begin{align*}
    				\sup_{0\leq t\leq T}\|\tilde{U}_{t}\|_{\mathrm{F}}^{2} & \leq2k+2\sup_{0\leq t\leq T}t\int_{0}^{t}\|{C}_{\tilde{Y}_{s}}^{-1}\mathbb{E}[\tilde{Y}_{s} a(s,U_s^{\top}\tilde{Y}_{s})^{\top}](I_{d\times d}-P_{\tilde{U}_s})\|_{\mathrm{F}}^{2}\,ds\\
    				&\leq2k+2\sup_{0\leq t\leq T}t\int_{0}^{t}\|{C}_{\tilde{Y}_{s}}^{-\frac{1}{2}}\mathbb{E}[{C}_{\tilde{Y}_{s}}^{-\frac{1}{2}}\tilde{Y}_{s} a(s,U_s^{\top}\tilde{Y}_{s})^{\top}](I_{d\times d}-P_{\tilde{U}_s})\|_{\mathrm{F}}^{2}\,ds\\
    				&\leq 2k+2\sup_{0\leq t\leq T}t\int_{0}^{t}\|{C}_{\tilde{Y}_{s}}^{-\frac{1}{2}}\|_{\mathrm{F}}^2\mathbb{E}[| a(s,U_s^{\top}\tilde{Y}_{s})|^2]\,ds\\    					
    				&\leq 2k+2\sup_{0\leq t\leq T}t\int_{0}^{t} 2\gamma k C_{\mathrm{lgb}}(1 + \mathbb{E}[|U_s^{\top}\tilde{Y}_{s}|^2] ) ds\\
    				& \leq 2k+4\gamma k C_{\mathrm{lgb}} \left(1 + 3k (3 \rho^2 +1) e^{3k} \right) T\\
    				&  \leq 2k+4\gamma C_{\mathrm{lgb}} \left(3 \rho^2 +2 \right)3k^2 e^{3k}  T \leq3k,
    			\end{align*}
    			where in the second inequality we have used the fact that $P_{\tilde{U}_{s}}$ is an orthogonal projector. Finally,
    			\begin{equation}\label{eq: U eta}
    				\begin{aligned}
    					\sup_{0\leq t\leq T}\|\tilde{U}_{t}-U_0\|^2_{\mathrm{F}}\leq &2 \gamma C_{\mathrm{lgb}} k \left(1 + 3k \left(3 \rho^2 +1 \right) e^{3k} \right) T
    					\leq  \frac{\eta^2}{2},
    				\end{aligned}
    			\end{equation}
    			which comes again from the definition of $T$ in \eqref{eq:T loc}.
    	\end{proof}
    \end{Lemma}
    
    We next prove that $F:C \to C$ is well-defined.

    \begin{Proposition}[Well-definiteness of the map $F$]\label{prop: F}
   Under Assumptions \ref{ass: initial value 2+eps}, \ref{ass: local-lipschitzianity}, and \ref{ass: linear-growth bound}, the map $F :  \ C \longrightarrow C$, defined in \eqref{eq: F}, is well-defined.
   \begin{proof}
   We need to prove that for each element $x\in C$, $F(x)$ is single valued and in $C$. We recall that the relation \eqref{eq: F} defines the map $F$ as follows: for a given $U$ we obtain the solution $\tilde{Y}$ of the SDE \eqref{eq: F_1}, which is now a standard Itô's SDE as $U$ is fixed; then we use $(U,\tilde{Y})$ to find a new function $\tilde{U}$ which solves the ODE defined by \eqref{eq: F_2} and set $F(U)= \tilde{U}$.

   Notice that in \eqref{eq: F_2} the input $U$ is used in the drift term, instead of the sought solution $\tilde{U}$. Moreover, $T$ being chosen as in \eqref{eq:T loc}, 
   ${C}_{\tilde{Y}_{t}}$ is always invertible in $[0,T]$ by Proposition \ref{prop: A inverse}, and $P_{\tilde{U}_{t}}^{\mathrm{row}}$ is always of full-rank $k$ in $[0,T]$ still thanks to Proposition \ref{prop: A inverse}. 
   
   First of all, given $U \in C$, Assumption \ref{ass: local-lipschitzianity} implies for all $t \geq 0$ and for all $y,z \in \mathbb{R}^k$ such that $|y|,|z| \leq n$ that 
   \begin{equation*}
   	| a(t,U_t^{\top}y) - a(t,U_t^{\top}z) | + \| b(t,U_t^{\top}y) - b(t,U_t^{\top}z) \|_{\mathrm{F}} \leq L_{\sqrt{3k}n} | U_t^{\top}y - U_t^{\top}z | \leq L_{\sqrt{3k}n}  \sqrt{3k} | y - z |,
   \end{equation*}
   and Assumption \ref{ass: linear-growth bound} implies
   \begin{equation}
   	|a(t,U_t^{\top}y)|^{2}+\|b(t,U_t^{\top}y)\|_{\mathrm{F}}^{2}\leq C_{\mathrm{lgb}}(1+|U_t^{\top}y|^{2}) \leq 3k C_{\mathrm{lgb}}(1+|y|^{2}), \quad  \forall y \in \mathbb{R}^k.
   \end{equation}
   Then, via 
   \cite[Theorem 3.4]{mao2007stochastic}, there exists a unique global solution of \eqref{eq: F_1} on $[0,T]$ which has almost surely continuous paths. Moreover, by Lemma \ref{lem: tilde U and Y in neigh} one has $\tilde{Y} \in\mathbb{D}_{\mathrm{sto}}^{\mathrm{loc}}$. Therefore, $\tilde{Y}$ has linearly independent components in $L^2(\Omega)$ (see Proposition \ref{prop: A inverse}). Consider this $\tilde{Y}$ fixed in \eqref{eq: F_2}, and define $G: [0,T] \times \mathbb{R}^{k \times d} \to \mathbb{R}^{k \times d}$ as
   \begin{equation}\label{eq: U inter}
   	 G(t, \tilde{U}_t):={C}_{\tilde{Y}_{t}}^{-1}\mathbb{E}[\tilde{Y}_{t}a(t,U_{t}^{\top}\tilde{Y}_{t})^{\top}](I_{d\times d}-P_{\tilde{U}_{t}}^{\mathrm{row}}),
   \end{equation}
   with $\tilde{U}(0) = U_0$ and $U_0$ orthonormal.
   First of all, one has
   	\begin{equation}\label{eq: orth U}
   		\tilde{U}_t  \frac{\mathrm{d}{\tilde{U}}_{t}}{\mathrm{d}t}^{\top} = \tilde{U}_t G(t,\tilde{U}_t)^{\top} = \tilde{U}_t (I_{d \times d}- P_{\tilde{U}_t}^{\mathrm{row}}) \mathbb{E}[a(t, U_t^{\top} \tilde{Y}_{t}) \tilde{Y}_t^{\top}] C_{\tilde{Y}_t}^{-1} = 0
   \end{equation}
   and, hence, 
      \begin{equation}
   	\begin{aligned}\label{eq: interm orth}
   	\tilde{U}_t  \tilde{U}_t^{\top} &= \int_0^t \mathrm{d} (\tilde{U}_s  \tilde{U}_s^{\top} )+ \tilde{U}_0  \tilde{U}_0^{\top} \\
   	&=  \int_0^t \frac{\mathrm{d}\tilde{U}_{s}}{\mathrm{d}s}  \tilde{U}_s^{\top} +  \tilde{U}_s \frac{\mathrm{d}\tilde{U}^{\top}_{s}}{\mathrm{d}s} + U_0 U_0^{\top} = I_{k \times k}, \quad \forall t \in [0,T].
\end{aligned}
\end{equation}
   Therefore, if $\tilde{U}_t$ exists, it has orthonormal rows for all $t$ and, hence, its norm is bounded by $\|\tilde{U}_{t}\|_{\mathrm{F}} \leq \sqrt{k}$. 
   On the other hand, the following bound on the difference between the deterministic orthogonal projectors for any two functions $(V,W) \in B_{\sqrt{3k}}(0) \cap B_{\frac{\eta}{\sqrt{2}}}(U_0)$ holds:
       \begin{equation*}
    	\begin{aligned}
    		\| P_{V_t}^{\mathrm{row}}-P_{W_t}^{\mathrm{row}} \|_{\mathrm{F}} \leq & \| V_t^{\top} (V_tV_t^{\top})^{-1} V_t - (W_t)^{\top} (W_t (W_t)^{\top})^{-1} W_t  \|_{\mathrm{F}} \\
    		\leq & \| V_t^{\top} (V_tV_t^{\top})^{-1} V_t - (W_t)^{\top} (V_tV_t^{\top})^{-1} V_t +(W_t)^{\top} (V_tV_t^{\top})^{-1} V_t - (W_t)^{\top} (W_t(W_t)^{\top})^{-1} V_t \\
    		&+ (W_t)^{\top} (W_t(W_t)^{\top})^{-1} V_t - (W_t)^{\top} (W_t (W_t)^{\top})^{-1} W_t  \|_{\mathrm{F}} \\
    		\leq & (2 \sqrt{k} \sqrt{3k}  ) \|W_t-V_t\|_{\mathrm{F}} + 3k  \|(V_tV_t^{\top})^{-1} - (W_t(W_t)^{\top})^{-1}\|_{\mathrm{F}} +   (2 \sqrt{k} \sqrt{3k}  ) \|W_t-V_t\|_{\mathrm{F}} \\
    		\leq & 4 \sqrt{3}k \|W_t-V_t\|_{\mathrm{F}} + 3k  \|(V_tV_t^{\top})^{-1} \left(W_tW_t^{\top} - V_tV_t^{\top}\right) (W_t(W_t)^{\top})^{-1}\|_{\mathrm{F}}  \\
    		\leq & (4 \sqrt{3}k +24k^2 \sqrt{3k})\|W_t-V_t\|_{\mathrm{F}}.
    	\end{aligned}
\end{equation*}
   
   From this we deduce that the right-hand side $G$ in \eqref{eq: U inter} is locally Lipschitz on $B_{\sqrt{3k}}(0) \cap B_{\frac{\eta}{\sqrt{2}}}(U_0)$; indeed by Cauchy-Schwarz inequality, linear-growth bound, and \cite[Lemma 3.5]{kazashi2021existence} one has
   \begin{equation*}
   	\begin{aligned} 
   		\|  G(t, V_t)-  G(t,W_t) \|_{\mathrm{F}} =& \| {C}_{\tilde{Y}_{t}}^{-1}\mathbb{E}[\tilde{Y}_{t}a(t,U_{t}^{\top}\tilde{Y}_{t})^{\top}](I_{d\times d}-P_{V_{t}}^{\mathrm{row}})  \\
   		&- {C}_{\tilde{Y}_{t}}^{-1}\mathbb{E}[\tilde{Y}_{t}a(t,U_{t}^{\top}\tilde{Y}_{t})^{\top}](I_{d\times d}-P_{W_{t}}^{\mathrm{row}}) \|_{\mathrm{F}} \\
   		= &\| {C}_{\tilde{Y}_{t}}^{-1}\mathbb{E}[\tilde{Y}_{t}a(t,U_{t}^{\top}\tilde{Y}_{t})^{\top}](P_{W_{t}}^{\mathrm{row}} -P_{V_{t}}^{\mathrm{row}})  \|_{\mathrm{F}} \\
   		= &\| {C}_{\tilde{Y}_{t}}^{-\frac{1}{2}}\mathbb{E}[ {C}_{\tilde{Y}_{t}}^{-\frac{1}{2}}\tilde{Y}_{t}a(t,U_{t}^{\top}\tilde{Y}_{t})^{\top}](P_{W_{t}}^{\mathrm{row}} -P_{V_{t}}^{\mathrm{row}})  \|_{\mathrm{F}} \\
   		\leq & \sqrt{2\gamma} \sqrt{C_{\mathrm{lgb}}(1 + \sqrt{3k} K(T))}  (4\sqrt{k}+24k^2)\sqrt{3k}	\| V_t- W_t \|_{\mathrm{F}} \\
   		\leq &  \sqrt{2\gamma} \sqrt{C_{\mathrm{lgb}}(1 + \sqrt{3k} K(T))}  (4\sqrt{k}+24k^2)\sqrt{3k}  \sup_{0 \leq t \leq T}	\| V_t- W_t \|_{\mathrm{F}},
   	\end{aligned}
   \end{equation*}
   where we recall that $\gamma :=\|{C}_{Y_0}^{-1}\|_{\mathrm{F}}$ and we notice that $ \sqrt{2\gamma} \sqrt{C_{\mathrm{lgb}}(1 + \sqrt{3k} K(T))} (4+24k^2)\sqrt{3k}$ depends on $\sqrt{3k}$, hence the radius of the ball $B_{\sqrt{3k}}(U_0)$.
   Therefore, via \cite[Corollary 2.1]{hartman2002ordinary} there exists a unique local solution $\tilde{U}$ of \eqref{eq: U inter} on an   interval $[0, T')$ such that $\tilde{U}_t \in B_{\sqrt{3k}}(0) \cap B_{\frac{\eta}{\sqrt{2}}}(U_0)$ for any $t \in [0, T')$. We show by contradiction that $T'=T$. Assume that $T' < T$, then either one has $\lim\limits_{t \to T'} \| \tilde{U}_t \|_{\mathrm{F}} = \sqrt{3k}$, or $\lim\limits_{t \to T'} \| \tilde{U}_t - U_0 \|_{\mathrm{F}} = \frac{\eta}{\sqrt{2}}$. But this is impossible since, on the one hand we have $\| \tilde{U}_t\|_{\mathrm{F}} = \sqrt{k}$ for all $t \in [0, T')$ by \eqref{eq: interm orth}, and, on the other hand, $\| \tilde{U}_t - U_0 \|_{\mathrm{F}} \leq \frac{\eta}{\sqrt{2}} - \varepsilon$ for all $t \in [0, T')$ and some $\varepsilon>0$ by \eqref{eq: U eta}, since $T'<T$. 
   
   Being the solution of an ODE, $\tilde{U}$ is absolutely continuous in $[0,T]$ and, from \eqref{eq: interm orth}, has orthonormal rows and, hence, $\|\tilde{U}_t\|_{\mathrm{F}} = \sqrt{k}$. Furthermore, by construction of $T$ and by Lemma \ref{lem: tilde U and Y in neigh} one has that $\tilde{U} \in \overline{B_{\frac{\eta}{\sqrt{2}}}(U_0)}$.
   Therefore, $\tilde{U}=F(U)$ is uniquely defined and $\tilde{U} \in C$. 
\end{proof}
    \end{Proposition}
    
    We continue by showing the continuity of $F$
    \begin{Proposition}[Continuity of $F$]\label{prop: F cont}
    Under Assumptions \ref{ass: initial value 2+eps}, \ref{ass: local-lipschitzianity}, and \ref{ass: linear-growth bound}, the map $F :  \ C \longrightarrow C$, defined in \eqref{eq: F}, is continuous.
    \begin{proof}
     Let us consider a sequence $U^n \in C$ which converges to $U\in\mathfrak{C}$. Then, from Lemma \ref{lem: C}  we have $U\in C$. 
	 Moreover, let us denote by $\tilde{Y}^n,\tilde{Y}$ the solutions of the SDE \eqref{eq:DLR-eq-Y} with frozen $U^n,U$, respectively, and let us denote $\tilde{U}^n=F(U^n)$ and $\tilde{U}=F(U)$, with $\tilde{U}^n, \tilde{U} \in C$ by Proposition \ref{prop: F}. Our aim is to prove that 
    $$ \lim\limits_{n \to \infty}\sup_{t\in[0,T]}\|\tilde{U}^n_t-\tilde{U}_t\|_{\mathrm{F}}=0.$$
    We do this in two steps. We first prove that $\tilde{Y}_n = F_1(U^n)$ converges to $\tilde{Y}$ in some sense. In our context, it is enough to prove that $\tilde{Y}^n$ converges to $\tilde{Y}$ uniformly in $[0,T]$ in $L^2(\Omega)$, as we will explain in detail later on. Then, we show that $\tilde{U}^n = F_2(U^n,\tilde{Y}^n)$ converges to $\tilde{U}$, where $F_1, F_2$ are defined in \eqref{eq: F}.

    First, fix $R>0$ and define the stopping time
    \begin{equation}\label{eq:tau_R}
    	\tau_{R,n}:=\inf\{t\ge0:\ \sup_{s\in [0,t]}|\tilde{Y}^n_s|\vee\sup_{s\in [0,t]}|\tilde{Y}_s|\ge R\}\wedge T.
    \end{equation}
    with the convention that $\inf \varnothing = T$. By Assumption \ref{ass: local-lipschitzianity}, on the ball $B_R:=\{y \in \mathbb{R}^k : |y| \leq R\}$ there exists a constant $L_R>0$ such that for all $t\in[0,T]$, for all $U \in C$, and for all $y,z\in B_R$,
    $$	|a(t,U^{\top}y)-a(t,U^{\top}z)|+\|b(t,U^{\top}y)-b(t,U^{\top}z)\|_{\mathrm{F}}\leq \sqrt{3k} L_{\sqrt{3k}R}|y-z|.$$
    	
    	Define the stopped processes $\tilde{Y}^{n,\tau_{R,n}}_t:=\tilde{Y}^n_{t\wedge\tau_{R,n}}$, $\tilde{Y}^{\tau_{R,n}}_t:=\tilde{Y}_{t\wedge\tau_{R,n}}$ and set
    	$\Delta_t:=\tilde{Y}^{n,\tau_{R,n}}_t-\tilde{Y}^{\tau_{R,n}}_t$.
    	For $t\in[0,T]$, one has that
    	\begin{equation*}
    	\begin{aligned}
    		\Delta_t
    		=& \int_0^t \big(U^n_s a(s,(U^n_s)^\top\tilde{Y}^n_s)-U_s a(s,U_s^\top\tilde{Y}_s)\big)  \mathbbm{1}_{\{s\le\tau_{R,n}\}} \,\mathrm{d}s \\
    		&+ \int_0^t \big(U^n_s b(s,(U^n_s)^\top\tilde{Y}^n_s)-U_s b(s,U_s^\top\tilde{Y}_s)\big)\,\mathbbm{1}_{\{s\le\tau_{R,n}\}}\mathrm{d}W_s.
    	\end{aligned}
    	\end{equation*}
    	Adding and subtracting crossed terms one gets
    	\begin{equation}\label{eq: Delta_t}
    		\begin{aligned}
    			\Delta_t
    			=& \int_0^t \left(U^n_s \left(a(s,(U^n_s)^\top\tilde{Y}^n_s)-a(s,(U^n_s)^\top\tilde{Y}_s)\right)
    			+ \big(U^n_s-U_s\big)a(s,U_s^\top\tilde{Y}_s)  \right) \mathbbm{1}_{\{s\le\tau_{R,n}\}} \,\mathrm{d}s \\
    			&+ \int_0^t  U^n_s\left(a(s,(U^n_s)^\top\tilde{Y}_s)-a(s,U_s^\top\tilde{Y}_s)\right) \mathbbm{1}_{\{s\le\tau_{R,n}\}} \,\mathrm{d}s \\
    			& +\int_0^t \left( U^n_s \left(b(s,(U^n_s)^\top\tilde{Y}^n_s)-b(s,(U^n_s)^\top\tilde{Y}_s)\right)
    			+ \left(U^n_s-U_s\right)b(s,U_s^\top\tilde{Y}_s)  \right) \mathbbm{1}_{\{s\le\tau_{R,n}\}} \,\mathrm{d}W_s. \\
    			&+ \int_0^t  U^n_s\left(b(s,(U^n_s)^\top\tilde{Y}_s)-b(s,U_s^\top\tilde{Y}_s)\right) \mathbbm{1}_{\{s\le\tau_{R,n}\}} \,\mathrm{d}W_s. \\
    		\end{aligned}
    	\end{equation}
    	By boundedness of $U^n$ and the Lipschitz bound on $B_R$, one gets
    	\begin{equation*}
    		\big|U^n_s\big(a(s,(U^n_s)^\top\tilde{Y}^n_s)-a(s,(U^n_s)^\top\tilde{Y}_s)\big)\mathbbm{1}_{\{s\le\tau_{R,n}\}}\big|
    	\le L_{\sqrt{3k}R} \sqrt{3k} |(U^n_s)^\top(\tilde{Y}^n_s-\tilde{Y}_s)\mathbbm{1}_{\{s\le\tau_{R,n}\}}| \le L_{\sqrt{3k}R} 3k|\Delta_s|.
    	\end{equation*}
    Moreover, using linear growth, for all $\omega \in \{s\le\tau_{R,n}\}$, it holds 
    	\begin{equation*}
    	\big|\big(U^n_s-U_s\big)a(s,U_s^\top\tilde{Y}_s)\mathbbm{1}_{\{s\le\tau_{R,n}\}}\big|
    	\le \|U^n_s-U_s\|_{\mathrm{F}}\,|a(s,U_s^\top\tilde{Y}_s)\mathbbm{1}_{\{s\le\tau_{R,n}\}}|
    	\le  \|U^n_s-U_s\|_{\mathrm{F}} \,C_{\mathrm{lgb}}(1+\sqrt{3k}R).
    		\end{equation*}
    and similarly we have 
        \begin{equation*}
        	\begin{aligned}
    	\big|U^n_s\big(a(s,(U^n_s)^\top\tilde{Y}_s)-a(s,U_s^\top\tilde{Y}_s)\big)\mathbbm{1}_{\{s\le\tau_{R,n}\}}\big| 
    	\le &L_{\sqrt{3k}R} \sqrt{3k}\big|(U^n_s-U_s)^\top\tilde{Y}_s\mathbbm{1}_{\{s\le\tau_{R,n}\}}\big| \\
    	\le & L_{\sqrt{3k}R} \sqrt{3k}\|U^n_s-U_s\|_{\mathrm{F}} R.
    	        	\end{aligned}
    	\end{equation*}
     By combining these three estimates we obtain for $s\le\tau_{R,n}$
        \begin{equation*}
        		\begin{aligned}
    	&\big|\big(U^n_s a(s,(U^n_s)^\top\tilde{Y}^n_s)-U_s a(s,U_s^\top\tilde{Y}_s)\big)\mathbbm{1}_{\{s\le\tau_{R,n}\}}\big| \\
    	\le & L_{\sqrt{3k}R} 3k |\Delta_s| + \left( C_{\mathrm{lgb}} (1+\sqrt{3k}R)+ L_{\sqrt{3k}R} R \sqrt{3k}  \right)  \|U^n_s-U_s\|_{\mathrm{F}}.
    	        	\end{aligned}
    	\end{equation*}
    	Similar relations hold for the diffusion coefficient with the same constants. Therefore, by taking the supremum over time and the expectation in \eqref{eq: Delta_t}, using Jensen's inequality, the Doob's martingale inequality and Itô's isometry, one obtains
      \begin{equation}\label{eq: Delta_t 2}
    	\begin{aligned}
    		&\mathbb{E}\Big[\sup_{t \in [0,T]}|\Delta_t|^2\Big] \\
    	\le & 4T \mathbb{E}\Big[\int_0^T \bigg(L_{\sqrt{3k}R}^2 9k^2|\Delta_s|^2 + \left( C_{\mathrm{lgb}} (1+\sqrt{3k}R)+ L_{\sqrt{3k}R} R  \sqrt{3k} \right)^2 \|U^n_s-U_s\|_{\mathrm{F}}^2\bigg) \mathbbm{1}_{\{s\le\tau_{R,n}\}}\mathrm{d}s\Big] \\
    		&+ 16\,\mathbb{E}\Big[\int_0^T \bigg(L_{\sqrt{3k}R}^2 9k^2|\Delta_s|^2 + \left( C_{\mathrm{lgb}} (1+\sqrt{3k}R)+ L_{\sqrt{3k}R} R \sqrt{3k} \right)^2 \|U^n_s-U_s\|_{\mathrm{F}}^2\bigg) \mathbbm{1}_{\{s\le\tau_{R,n}\}}\mathrm{d}s\Big]\\
    		\le& (4T+16) \left[ L_{\sqrt{3k}R}^2 9k^2 \int_0^T \mathbb{E}\big[ |\Delta_s|^2\big]\,\mathrm{d}s + \left( C_{\mathrm{lgb}} (1+\sqrt{3k}R)+ L_{\sqrt{3k}R} R \sqrt{3k} \right)^2 \int_0^T  \|U^n_s-U_s\|_{\mathrm{F}}^2 \mathrm{d}s\right] \\
    		\le& (4T+16) \left[ L_{\sqrt{3k}R}^2 9k^2 \int_0^T \mathbb{E}\big[ \sup_{r \in [0,s]} |\Delta_r|^2\big]\,\mathrm{d}s + \left( C_{\mathrm{lgb}} (1+\sqrt{3k}R)+ L_{\sqrt{3k}R} R \sqrt{3k} \right)^2 T  \sup_{t \in [0,T]} \|U^n_t-U_t\|_{\mathrm{F}}^2\right].
    	\end{aligned}
    	\end{equation}

    	By Gronwall's inequality, from \eqref{eq: Delta_t 2} there exists a constant $K_R>0$ dependent on $R$ such that
    	\begin{equation}\label{eq: sup Y_tau_n}
    	\mathbb{E}\Big[\sup_{t \in [0,T]}|\tilde{Y}^{n,\tau_{R,n}}_t-\tilde{Y}^{\tau_{R,n}}_t|^2\Big]
    	\le K_R \sup_{t \in [0,T]} \|U^n_t-U_t\|_{\mathrm{F}}^2. 
    	\end{equation}
    	Under Assumptions \ref{ass: initial value 2+eps} and \ref{ass: linear-growth bound}, for the positive constant $\tilde{K}(T)$ defined in \eqref{eq: tilde K p + eps} with $p=2$, one has
    	$\sup\limits_n \mathbb{E}[\sup\limits_{t \in [0,T]}|\tilde{Y}^n_t|^2]\le \tilde{K}(T)$ and $\mathbb{E}[\sup\limits_{t \in [0,T]}|\tilde{Y}_t|^2]\le \tilde{K}(T)$.
    	Then, Chebyshev's inequality implies that 
    	\begin{equation*}
    		\begin{aligned}
    	\sup_n \mathbb{P}(\tau_{R,n}<T) &= \sup_n \mathbb{P}\left( \left(\inf\{t\ge0:\ \sup_{s\in [0,t]}|\tilde{Y}^n_s|\vee\sup_{s\in [0,t]}|\tilde{Y}_s|\ge R\}\wedge T \right) <T\right) \\
    	& \leq \sup_n \mathbb{P}(\inf\{t\ge0:\ \sup_{s\in [0,t]}|\tilde{Y}^n_s|\ge R\} <T) +  \mathbb{P}(\inf\{t\ge0:\ \sup_{s\in [0,t]}|\tilde{Y}_s|\ge R\} <T) \\
    	& \leq\frac{\sup\limits_n \mathbb{E}[\sup\limits_{t \in [0,T]}|\tilde{Y}^n_t|^2] + \mathbb{E}[\sup\limits_{t \in [0,T]}|\tilde{Y}_t|^2]}{R^2}
    	\le \frac{2\tilde{K}(T)}{R^2}.
    	\end{aligned}
    	\end{equation*}
        Notice that the final right-hand side is independent of $n$.
    	Now fix an arbitrary $\overline{\varepsilon}>0$. Choose $R>0$ so large that
    	\begin{equation*}
    	\frac{2\tilde{K}(T)}{R^2} < \frac{\overline{\varepsilon}}{2}.
    	\end{equation*}
    	With this $R$ fixed, \eqref{eq: sup Y_tau_n} and Chebyshev's inequality yield
    	\begin{equation*}
    	\mathbb{P}\Big(\sup_{t \in [0,T]}|\tilde{Y}^{n,\tau_{R,n}}_t-\tilde{Y}^{\tau_{R,n}}_t|>\delta\Big)
    	\le \frac{K_R}{\delta^2} \sup_{t \in [0,T]} \|U^n_t-U_t\|_{\mathrm{F}}^2.
    		\end{equation*}
    	Now, let $\delta >0$ arbitrary and, by hypothesis of convergence of $U^n_t$ to $U_t$, one can choose $N_{\varepsilon,\delta}$ such that for all $n\ge N_{\varepsilon,\delta}$,
    	\begin{equation*}
    	\frac{K_R}{\delta} \sup_{t \in [0,T]} \|U^n_t-U_t\|_{\mathrm{F}}^2 < \frac{\overline{\varepsilon}}{2}.
    	\end{equation*}
    	Using all the above relations, for $n\ge N_{\varepsilon,\delta}$ one obtains
    	\begin{equation}\label{eq: Y_n prop}
    	\begin{aligned}
    		\mathbb{P}\Big(\sup_{t \in [0,T]}|\tilde{Y}^n_t-\tilde{Y}_t|>\overline{\varepsilon}\Big)
    		&\le \mathbb{P}(\tau_{R,n}<T) + \mathbb{P}\Big(\sup_{t \in [0,T]}|\tilde{Y}^{n,\tau_{R,n}}_t-\tilde{Y}^{\tau_{R,n}}_t|>\overline{\varepsilon}\Big) \\
    		&< \frac{\overline{\varepsilon}}{2} + \frac{\overline{\varepsilon}}{2} =\overline{\varepsilon},
    	\end{aligned}
    		\end{equation}
	hence $\tilde{Y}^n$ converges to $\tilde{Y}$ uniformly on $[0,T]$ in probability.
    	
        Now we want to exploit the uniform convergence in probability to prove the convergence of the expectation terms in \eqref{eq: F_1}. One can prove for $\tilde{Y},\tilde{Y}^{n}$ a similar result to Lemma \ref{lem: DLR Gronwall - lg}, i.e.\
        $\mathbb{E}[\sup\limits_{t \in [0,T]}|\tilde{Y}_t|^{2+\varepsilon}]\le C_{T,\varepsilon} \left( 1+\mathbb{E}[|Y_0|^{2+\varepsilon}]\right)$ and  $\mathbb{E}[\sup\limits_{t \in [0,T]}|\tilde{Y}^n_t|^{2+\varepsilon}]\le C_{T,\varepsilon} \left( 1+\mathbb{E}[|Y_0|^{2+\varepsilon}]\right)$, with $ C_{T,\varepsilon}$ independent of $n$.
		By using this uniform $2+\varepsilon$ moment bound for $\tilde{Y}^n_t$ and $\tilde{Y}_t$, one has for all $K>0$
		\begin{equation*}
			\mathbb{E}[\sup_{t \in [0,T]}|\tilde{Y}^n_t-\tilde{Y}_t|^2 \mathbbm{1}_{\{\sup_{t \in [0,T]}|\tilde{Y}^n_t-\tilde{Y}_t|^2 > K\}} ] \leq \mathbb{E}[\sup_{t \in [0,T]}|\tilde{Y}^n_t-\tilde{Y}_t|^{2+ \varepsilon}  \frac{1}{K^{\frac{\varepsilon}{2} }} ]  \leq \frac{2^{2 + \varepsilon} C_{T,\varepsilon} \left( 1+\mathbb{E}[|Y_0|^{2+\varepsilon}]\right)}{K^{\frac{\varepsilon}{2}}},
		\end{equation*}
		where the right-hand side of the previous relation is independent of $n$. This relation implies the uniform integrability in $n$ of $\sup_{t \in [0,T]} |\tilde{Y}^n_t - \tilde{Y}_t|^2$.
		Then, Vitali's Theorem \cite[Theorem 16.6]{schilling2017measures} yields that the quantity $\mathbb{E}[\sup_{t\in[0,T]}|\tilde{Y}^n_t - \tilde{Y}_t|^2]$ goes to $0$ for $n$ going to $+\infty$, which implies the convergence of $\tilde{Y}^n$ to $\tilde{Y}$ in $L^2(\Omega)$ uniformly in $[0,T]$. 
		
		Then, denoting $C_{\tilde{Y}^n_t}=\mathbb{E}[\tilde{Y}^n_t(\tilde{Y}^n_t)^{\top}]$ and  $C_{\tilde{Y}_t}=\mathbb{E}[\tilde{Y}_t\tilde{Y}_t^{\top}]$, by Cauchy-Schwarz inequality one has that
    	\begin{equation*}
    		\sup_{t\in[0,T]}\|C_{\tilde{Y}^n_t}-C_{\tilde{Y}_t}\|_{\mathrm{F}} \leq 2 \sqrt{\tilde{K}(T)} \sqrt{ \sup_{t\in[0,T]} \mathbb{E}[|\tilde{Y}^n_t-\tilde{Y}_t|^2]},
    	\end{equation*}
    	and, hence, by the convergence of $\tilde{Y}^n$ to $\tilde{Y}$ in $L^2(\Omega)$ uniformly in $[0,T]$ we obtain
    	\begin{equation}\label{eq: gramian convergence}
    		\lim\limits_{n \to \infty}\sup_{t\in[0,T]}\|C_{\tilde{Y}^n_t}-C_{\tilde{Y}_t}\|_{\mathrm{F}} =0.
    	\end{equation}
    	Moreover, using Assumption \ref{ass: local-lipschitzianity} and a similar truncation argument to the one employed to show $\tilde{Y}^{n} \to \tilde{Y}$, we have that 
    	\begin{equation}\label{eq: drift term convergence}
    	\lim\limits_{n\to \infty} \sup_{t\in[0,T]}\Big\|\mathbb{E}[\tilde{Y}^n_t a(t,(U^n_t)^\top\tilde{Y}^n_t)^\top] - \mathbb{E}[\tilde{Y}_t a(t,U_t^\top\tilde{Y}_t)^\top]\Big\|_{\mathrm{F}}=0.
    	\end{equation}

    As $U^n$ and $U$ belong to $C$, then we have $(U^n,\tilde{Y}^n) \in\mathbb{D}_{\mathrm{det}}^{\mathrm{loc}}\times\mathbb{D}_{\mathrm{sto}}^{\mathrm{loc}}$ and $(U,\tilde{Y}) \in\mathbb{D}_{\mathrm{det}}^{\mathrm{loc}}\times\mathbb{D}_{\mathrm{sto}}^{\mathrm{loc}}$ by Lemma \ref{lem: tilde U and Y in neigh}.
	    Then, Proposition \ref{prop: A inverse} yields that for all $t\in[0,T]$ and for all $n$, $C_{\tilde{Y}^n_t}$ and $C_{\tilde{Y}_t}$ are invertible and
    	their inverses are uniformly bounded:
    \begin{equation}\label{eq: inverse bound}
    	\sup_{t\in[0,T]}\|C_{\tilde{Y}^n_t}^{-1}\|_{\mathrm{F}} \le 2\gamma ,\qquad \sup_{t\in[0,T]}\|C_{\tilde{Y}_t}^{-1}\|_{\mathrm{F}}  \le 2\gamma.
    		\end{equation}
    	Thanks to \eqref{eq: inverse bound}, matrix inversion is a continuous map on the set $\{M \in \mathbb{R}^{k \times k} :\ \|M^{-1}\|_{\mathrm{F}}  \le 2\gamma \}$. Indeed, given $A,B \in \mathbb{R}^{k \times k}$ invertible matrices, then
    	\begin{equation}\label{eq: A-B -1}
    		A^{-1}-B^{-1}= B^{-1}\left(B-A\right)A^{-1}.
    	\end{equation} 
    	Therefore, for all $n$ and for all $t \in [0,T]$ one has
    	\begin{equation*}
    	\|C_{\tilde{Y}^n_t}^{-1}-C_{\tilde{Y}_t}^{-1}\|_{\mathrm{F}} \leq  \|C_{\tilde{Y}_t}^{-1}\|_{\mathrm{F}} 	\|C_{\tilde{Y}_t}-C_{\tilde{Y}^n_t}\|_{\mathrm{F}}  \|C_{\tilde{Y}^n_t}^{-1}\|_{\mathrm{F}} \leq 4\gamma^2 	\|C_{\tilde{Y}_t}-C_{\tilde{Y}^n_t}\|_{\mathrm{F}}
    	\end{equation*}
    	which, together with \eqref{eq: gramian convergence}, implies that
    	\begin{equation}\label{eq: inverse convergence}
    	\lim\limits_{n \to \infty}\sup_{t\in[0,T]}\|C_{\tilde{Y}^n_t}^{-1}-C_{\tilde{Y}_t}^{-1}\|_{\mathrm{F}}=0.
    		\end{equation}
    	
    	Now we use the previous relation to prove convergence of $F(U^n)$ to $F(U)$, whose fact yields the continuity of $F$. For the sake of notation, let us define
    \begin{equation*}
    	A^n(t):=C_{\tilde{Y}^n_t}^{-1}\,\mathbb{E}\big[\tilde{Y}^n_t a\big(t,(U^n_t)^\top\tilde{Y}^n_t\big)^\top\big],\qquad A(t):=C_{\tilde{Y}_t}^{-1}\,\mathbb{E}\big[\tilde{Y}_t a(t,U_t^\top\tilde{Y}_t)^\top\big].
    		\end{equation*}
    	We can bound the difference $A^n(t)-A(t)$ in Frobenius norm as
    \begin{equation*}
    	\begin{aligned}
    		\|A^n(t)-A(t)\|_{\mathrm{F}}
    		&\le \|C_{\tilde{Y}^n_t}^{-1}-C_{\tilde{Y}_t}^{-1}\|_{\mathrm{F}} \cdot \Big\|\mathbb{E}\big[\tilde{Y}^n_t a(t,(U^n_t)^\top\tilde{Y}^n_t)^\top\big]\Big\|_{\mathrm{F}} \\
    		&\quad + \|C_{\tilde{Y}_t}^{-1}\|_{\mathrm{F}} \cdot \Big\|\mathbb{E}\big[\tilde{Y}^n_t a(t,(U^n_t)^\top\tilde{Y}^n_t)^\top\big]
    		- \mathbb{E}\big[\tilde{Y}_t a(t,U_t^\top\tilde{Y}_t)^\top\big]\Big\|_{\mathrm{F}}.
    	\end{aligned}
    	\end{equation*}
    	Using the uniform second-moment bound for the stochastic basis, the boundedness of the deterministic one, Assumption \ref{ass: linear-growth bound}, assumptions on the Gramian, and Cauchy-Schwarz inequality, one has
    	\begin{equation*}
    		\begin{aligned}
    	\sup_{t\in[0,T]}\|A^n(t)-A(t)\|_{\mathrm{F}} 
		\le  &\sqrt{ \tilde{K}(T) C_{\mathrm{lgb}} (1 + \sqrt{3k} \tilde{K}(T)) }\sup_{t\in[0,T]}\|C_{\tilde{Y}^n_t}^{-1}-C_{\tilde{Y}_t}^{-1}\|_{\mathrm{F}} \\
    	&+ 2\gamma \sup_{t\in[0,T]} \Big\|\mathbb{E}[\tilde{Y}^n_t a(t,(U^n_t)^\top\tilde{Y}^n_t )^\top]-\mathbb{E}[\tilde{Y}_t a(t ,U_t^\top\tilde{Y}_t)^\top]\Big\|_{\mathrm{F}}.
    	\end{aligned}
    	\end{equation*}
    	 Consequently, from \eqref{eq: drift term convergence} and \eqref{eq: inverse convergence} one has
    \begin{equation}\label{eq: A conv}
    	\lim\limits_{n \to \infty} \sup_{t\in[0,T]}\|A^n(t)-A(t)\|_{\mathrm{F}}=0.
    		\end{equation}
    	
    Now, $U_0 U_0^{\top} = I_{k \times k}$ and $\| (U_0 U_0^{\top})^{-1} \|_{\mathrm{F}} = \sqrt{k}$. Then, thanks to Proposition \ref{prop: A inverse}, the orthogonality of the rows of $U$, and \eqref{eq: A-B -1}, one can write 
    \begin{equation*}
    	\begin{aligned}
    		\| P_{\tilde{U}_t}^{\mathrm{row}}-P_{\tilde{U}^n_t}^{\mathrm{row}} \|_{\mathrm{F}}
    		\leq &2\|\tilde{U}_t^n-\tilde{U}_t\|_{\mathrm{F}}.
    	\end{aligned}
    \end{equation*}
    
    Finally, define the difference $E_t:=\tilde{U}^n_t-\tilde{U}_t$. Then $E_t$ satisfies the following ODE
    	\begin{equation*}
    	\dot{E}_t = (A^n(t)-A(t))(I-P_{\tilde{U}^n_t}^{\mathrm{row}}) + A(t)\big(P_{\tilde{U}_t}^{\mathrm{row}}-P_{\tilde{U}^n_t}^{\mathrm{row}}\big),
    		\end{equation*}
    with $E_0=0$.
    Passing to the norm and to the integral over time one gets 
    \begin{equation*}
    	\|E_t\|_{\mathrm{F}} \le \int_0^t \|A^n(s)-A(s)\|_{\mathrm{F}} \,\mathrm{d}s + C(T) \int_0^t \|E_s\|_{\mathrm{F}} \,\mathrm{d}s,
    		\end{equation*}
    		where $C(T):= 2  \sqrt{ 2 \gamma \tilde{K}(T) C_{\mathrm{lgb}} (1 + \sqrt{3k} \tilde{K}(T)) }$. Thus, Gronwall's inequality yields
    		\begin{equation*}
    		\sup_{t\in[0,T]}\|E_t\|_{\mathrm{F}}
    		\le T\Big(\sup_{t\in[0,T]} \|A^n(t)-A(t)\|_{\mathrm{F}}\Big)\exp(C(T) T),
    	\end{equation*}
    	Using \eqref{eq: A conv} we get 
    \begin{equation*}
    \lim\limits_{n \to \infty}	\sup_{t\in[0,T]}\|\tilde{U}^n_t-\tilde{U}_t\|_{\mathrm{F}} = 0,
    		\end{equation*}
    i.e.\ the continuity of the map $F$.
    \end{proof}
    \end{Proposition}
    
   Via Ascoli-Arzelà Theorem (see e.g.\ \cite[Chapter III Theorem 3.1]{lang2012real}), one can finally prove the compactness of $F$.
  \begin{Proposition}[Compactness of $F$]\label{prop: compact F}
  	Under Assumptions \ref{ass: initial value 2+eps}, \ref{ass: local-lipschitzianity}, and \ref{ass: linear-growth bound}, the map $F :  \ C \longrightarrow C$, defined as in \eqref{eq: F}, is compact.
  	\begin{proof}
  	    We already know that $F$ is continuous from Proposition \ref{prop: F cont}.  We only need to prove the relative compactness of $F(C)$. We do this by showing that the hypotheses of Ascoli-Arzelà Theorem (see e.g.\ \cite[Chapter III Theorem 3.1]{lang2012real}) are satisfied. We first show that $F (C)$ is equi-bounded: indeed, for all $\tilde{U} \in F(C)$ (and respective $\tilde{Y}$) by properties of matrix norms and Cauchy-Schwarz inequality one has
  		\begin{align*}
  			\sup_{0\leq t \leq T}  \|\tilde{U}_t\|_{\mathrm{F}} &\leq \|U_0\|_{\mathrm{F}} + \sup_{0\leq t \leq T}  \left\|\int_0^t C_{\tilde{Y}_s}^{-1} \mathbb{E}[\tilde{Y}_s a(s, U_s^{\top} \tilde{Y}_s)^{\top}](I_{d \times d}-P_{\tilde{U}_{s}}^{\mathrm{row}} ) \mathrm{d}s\right\|_{\mathrm{F}} \\
  			&\leq \|U_0\|_{\mathrm{F}} + \sup_{0\leq t \leq T}  \int_0^t \left\| C_{\tilde{Y}_s}^{-1} \mathbb{E}[\tilde{Y}_s a(s, U_s^{\top} \tilde{Y}_s)^{\top}](I_{d \times d}-P_{\tilde{U}_{s}}^{\mathrm{row}} ) \right\|_{\mathrm{F}} \mathrm{d}s \\
  			&\leq \|U_0\|_{\mathrm{F}} + \int_0^T \sqrt{2\gamma k} \|C_{\tilde{Y}_s}^{-\frac{1}{2}}\mathbb{E}[\tilde{Y}_s a(s,U_s^{\top} \tilde{Y}_s)]\|_{\mathrm{F}} |I_{d \times d}- P_{\tilde{U}_{s}}^{\mathrm{row}}| \mathrm{d}s \\
  			&\leq \|U_0\|_{\mathrm{F}} + \int_0^T \sqrt{2\gamma k} \sqrt{\mathbb{E}[|a(s,U_s^{\top} \tilde{Y}_s)|^2]} |I_{d \times d}- P_{\tilde{U}_{s}}^{\mathrm{row}}| \mathrm{d}s \\
  			&\leq \|U_0\|_{\mathrm{F}} + \int_0^T \sqrt{2\gamma k}  \sqrt{C_{\mathrm{lgb}}(1 + 3k \tilde{K}(T))} \mathrm{d}s \\
  			&\leq \sqrt{k} + T \cdot \sqrt{2\gamma k} \sqrt{C_{\mathrm{lgb}}(1 + 3k \tilde{K}(T))},
  		\end{align*}
  		where the right-hand side in the last line is independent of $U$.
		
  		Furthermore, let us prove that the set $F(C)$ is equi-continuous, i.e.\ $\forall t \in [0,T],$
  		$\forall \varepsilon > 0$, $\exists \delta>0$ : $\forall s$ : $|s| < \delta \Rightarrow \|\tilde{U}(t+s) - \tilde{U}(t)\|_{\mathrm{F}} < \varepsilon$, $\forall \tilde{U} \in F(C)$. First, by Proposition \ref{prop: F} for all $\tilde{U} \in F(C)$ there exists $U \in C$ and $\tilde{Y}=F_1(U)$ such that $\tilde{U} = F_2(U,\tilde{Y})$, i.e.\ 
        \begin{equation*}
            \frac{\mathrm{d}\tilde{U}_t}{\mathrm{d}t} = C_{\tilde{Y}_t}^{-1} \mathbb{E}\left[\tilde{Y}_t a(t,U_t^{\top}\tilde{Y}_t)^{\top}\right] \left(I_{d \times d} - P_{\tilde{U}_t}^{\mathrm{row}}\right).
        \end{equation*}
        Then, similarly to the previous computations, one has 
  		\begin{equation}\label{eq: equi}
  			\begin{aligned}
  				\sup_{t \in [0,T]} 	\|\tilde{U}_{t+s} -\tilde{U}_{t}\|_{\mathrm{F}}
  				& \leq \sup_{t \in [0,T]}  \left\|\int_t^{t+s} C_{\tilde{Y}_r}^{-1} \mathbb{E}[\tilde{Y}_r a(r, U_r^{\top} \tilde{Y}_r)^{\top}](I_{d \times d}- P_{\tilde{U}_{r}}^{\mathrm{row}}) dr\right\|_{\mathrm{F}} \\
  				& \leq\sup_{t \in [0,T]} \int_t^{t+s} \|C_{\tilde{Y}_r}^{-1}\mathbb{E}[\tilde{Y}_r a(r, U_r^{\top} \tilde{Y}_r)^{\top}]\|_{\mathrm{F}}|(I_{d \times d}- P_{\tilde{U}_{r}}^{\mathrm{row}})| dr \\
  				& \leq\sup_{t \in [0,T]} \int_t^{t+s} \|C_{\tilde{Y}_r}^{-\frac{1}{2}}\|_{\mathrm{F}} \sqrt{\mathbb{E}[|a(r, U_r^{\top} \tilde{Y}_r)|^2]} |(I_{d \times d}- P_{\tilde{U}_{r}}^{\mathrm{row}})| dr \\
  				& \leq \sup_{t \in [0,T]}  \int_t^{t+s}  \sqrt{2\gamma k} \sqrt{C_{\mathrm{lgb}}(1 +3 k \tilde{K}(T))} dr\\
  				&\leq  s\sqrt{2\gamma k} \sqrt{C_{\mathrm{lgb}}(1 + 3k \tilde{K}(T))} .
  			\end{aligned}
  		\end{equation}
  		The final upper-bound in \eqref{eq: equi} is independent of the chosen $\tilde{U}$.
  		Therefore, via defining $\delta = \frac{\varepsilon}{ \sqrt{2\gamma k} \sqrt{C_{\mathrm{lgb}}(1 + 3k \tilde{K}(T))} }$, we get that $F$ is equicontinuous. 
  		Thus, $F(C)$ is relatively compact, i.e.\ the map $F$ is compact. 
  	\end{proof}
  \end{Proposition}
  
  We can finally state our result of local existence.
	\begin{Theorem}\label{thm: well-posedness locl-linear-growth}
		Let Assumptions \ref{ass: initial value 2+eps}, \ref{ass: local-lipschitzianity}, and \ref{ass: linear-growth bound} hold. Moreover, assume that the initial conditions $(U_0,Y_0 )$ are such that $U_0  \in \mathbb{R}^{k\times d}$ is a matrix with orthonormal rows and that the components of \(Y_0\) are linearly independent in \(L^2(\Omega)\) with $\rho^2:=\|Y_0\|_{[L^{2}(\Omega)]^{k}}^{2}$ and $\gamma :=\|C_{Y_0}^{-1}\|_{\mathrm{F}}$. Then there exists a strong DO solution $(U,Y)$ of \eqref{eq:SDE-int} in $[0,T]$ with $T$ defined in \eqref{eq:T loc}.
		\begin{proof}
			 Consider $C$ defined as in \eqref{eq: C} and the map $F:C \to C$ defined as in \eqref{eq: F}. By Proposition \ref{prop: F} we know that $F$ is well defined and via Proposition \ref{prop: compact F} that $F$ is a compact map. By Schauder Theorem \cite[Theorem 8.8]{deimling2013nonlinear} there exists a fixed point $U \in C$ such that $F(U) = U$, which has orthonormal rows by construction of \eqref{eq: F} (as seen in relation \eqref{eq: interm orth}), and a unique $Y$ such that \eqref{eq:DLR-eq-Y} is satisfied for this fixed $U$. Then, by construction $(U,Y)$ is a strong DO solution of \eqref{eq:SDE-int}.
        \end{proof}
    \end{Theorem}

In case of elliptic diffusion, the strong DO solution exists globally. We consider the notation $A \succ B$ (respectively $A \succeq B$) with $A,B$ square matrices to indicate that $A-B$ is positive definite (respectively positive semidefinite).
\begin{Assumption}[Elliptic diffusion]\label{ass: diff}
	There exists a positive constant $\sigma_{B}$ such that $$b(t, x)b(t,x)^{\top}\succeq\sigma_{B} \cdot I_{d \times d} \succ 0,$$
	for all $t \in [0,+\infty)$ and for all $ x \in \mathbb{R}^d$. 
\end{Assumption}

\begin{Theorem}\label{thm: global well-posedness ll}
	Suppose that the hypotheses of Theorem \ref{thm: well-posedness locl-linear-growth} are satisfied and that Assumption \ref{ass: diff} holds. Then the strong DO solution exists for all $t \in [0,+\infty)$.
	\begin{proof}
		Under Assumptions \ref{ass: linear-growth bound} and \ref{ass: diff} and without requiring any hypothesis on the Lipschtzianity of the coefficients, \cite[Proposition 4.5]{kazashi2025dynamical} states that the Gramian $C_{Y_t}$ is strictly positive for all $t \geq 0$. Then, the proof of global existence follows verbatim to \cite[Theorem 4.6]{kazashi2025dynamical}. Indeed, suppose that there exists an explosion time $T_{e}$ satisfying $T_{e} <+ \infty$, then by \cite[Proposition 4.1 ]{kazashi2025dynamical} it follows that $\lim\limits_{t \to T_e} \|C_{Y_t}^{-1}\|_{\mathrm{F}}^2 = + \infty$. By this is a contradiction to the fact that $C_{Y_t} \succ 0$ for all $t\geq 0$, and, hence, $T_e= + \infty$.
	\end{proof}
\end{Theorem}
    
    \section{Uniqueness result under logarithmic growth of the local Lipschitz constant}\label{sec: ll uniqueness}
    
    The result in Theorem \ref{thm: well-posedness locl-linear-growth} only shows existence of a local strong DO solution. When comparing two solutions of \eqref{eq:DLR-eq-U}-\eqref{eq:DLR-eq-Y}, eventually restricted to the ball of radius $n$, the Lipschitz constant $L_n$ would appear.
    
    To be able to conclude that any two solutions coincide by some Gronwall type argument, we need the Lipschitz constant to grow not too fast in $n$, and compensate its growth with the tail bounds on the distribution of the stochastic basis. In this section we work out this argument assuming sub-logarithmic growth of $L_n$.

\begin{Assumption}[Logarithmic growth of $L_n$]\label{ass: log growth L_n}
	We assume that there exists a positive constant $C_{\mathrm{loc}}$ such that
	\begin{equation}
		L_n \leq C_{\mathrm{loc}}\sqrt{\log n}, \qquad n\geq2.
	\end{equation}
\end{Assumption}

Under the further Assumption \ref{ass: log growth L_n}, we can now show uniqueness of a strong DO solution.

\begin{Proposition}[Pathwise uniqueness]\label{prop: uniqueness loc}
	Suppose the hypotheses of Theorem \ref{thm: well-posedness locl-linear-growth} are satisfied and that Assumption \ref{ass: log growth L_n} holds. Then the pair $(U,Y)$, strong solution of \eqref{eq:SDE-int} in $[0,T]$ with $T$ defined in \eqref{eq:T loc}, is pathwise unique.
	
	\begin{proof}
		Consider two strong DO solutions $(U^1,Y^1)$ and $(U^2,Y^2)$ satisfying \eqref{eq:DLR-eq-U}--\eqref{eq:DLR-eq-Y} in the same interval $[0,T]$, with the same Brownian motion $W_t$, and the same initial data $(U_0,Y_0)$. Define the differences $\Delta U_t:=U^1_t-U^2_t$ and $\Delta Y_t:=Y^1_t-Y^2_t$, and the following functional
		\begin{equation}\label{eq:D_t}
			D_t:=\|\Delta U_t\|_{\mathrm{F}}^2+|\Delta Y_t|^2.
		\end{equation}
		We want to show that \eqref{eq:D_t} is null in expectation, which would imply that $\Delta U_t$ and $\Delta Y_t$ are a.s.\ zero. Therefore, $U^2$ would coincide with $U^1$ and $Y^2$ would be a modification of $Y^1$ and, since a strong DO solution has a.s.\ continuous paths, $Y^1$ and $Y^2$ would be indistinguishable. To prove this property, we will evaluate the difference in $L^2$ between the two DO solutions and conclude via Gronwall's lemma.
		
		First, notice that $U^1_t,U^2_t$ satisfy the following ODEs
		\begin{equation*}
			\begin{aligned}
				\frac{\mathrm{d}U^1_t}{\mathrm{d}t}&=C_{Y^1_t}^{-1}\mathbb{E}\!\left[Y^1_ta\!\left(t,(U^1_t)^\top Y^1_t\right)^\top\right](I_{d\times d}-P_{U^1_t}^{\mathrm{row}})=:f(t,U^1_t,Y^1_t),\\
				\frac{\mathrm{d}U^2_t}{\mathrm{d}t}&=C_{Y^2_t}^{-1}\mathbb{E}\!\left[Y^2_ta\!\left(t,(U^2_t)^\top Y^2_t\right)^\top\right](I_{d\times d}-P_{U^2_t}^{\mathrm{row}})=:f(t,U^2_t,Y^2_t),
			\end{aligned}
		\end{equation*}
		and, hence,
		\begin{equation}\label{eq: Delta U_t}
			\frac{\mathrm{d}}{\mathrm{d}t}\Delta U_t=f(t,U^1_t,Y^1_t)-f(t,U^2_t,Y^2_t).
		\end{equation}
		On the other hand, the stochastic bases $Y^1_t,Y^2_t$ satisfy, respectively,
		\begin{equation*}
			\begin{aligned}
			\mathrm{d}Y^1_t=&U^1_ta\!\left(t,(U^1_t)^\top Y^1_t\right)\mathrm{d}t+U^1_tb\!\left(t,(U^1_t)^\top Y^1_t\right)\mathrm{d}W_t,\\ 
			\mathrm{d}Y^2_t=&U^2_ta\!\left(t,(U^2_t)^\top Y^2_t\right)\mathrm{d}t+U^2_tb\!\left(t,(U^2_t)^\top Y^2_t\right)\mathrm{d}W_t,
			\end{aligned}
		\end{equation*}
		so their difference satisfies
		\begin{equation}\label{eq: Delta Y_t}
			\mathrm{d}\Delta Y_t=\left[U^1_ta\!\left(t,(U^1_t)^\top Y^1_t\right)-U^2_ta\!\left(t,(U^2_t)^\top Y^2_t\right)\right]\mathrm{d}t+\left[U^1_tb\!\left(t,(U^1_t)^\top Y^1_t\right)-U^2_tb\!\left(t,(U^2_t)^\top Y^2_t\right)\right]\mathrm{d}W_t.
		\end{equation}
		Via It\^o's formula, the functional $D_t$ satisfies
		\begin{equation}\label{eq: Ito D uniqueness}
			\begin{aligned}
				\mathrm{d}D_t=&2\left\langle\Delta U_t,f(t,U^1_t,Y^1_t)-f(t,U^2_t,Y^2_t)\right\rangle_{\mathrm{F}}\mathrm{d}t+2\left\langle\Delta Y_t,U^1_ta\!\left(t,(U^1_t)^\top Y^1_t\right)-U^2_ta\!\left(t,(U^2_t)^\top Y^2_t\right)\right\rangle\mathrm{d}t\\
				+&\left\|U^1_tb\!\left(t,(U^1_t)^\top Y^1_t\right)-U^2_tb\!\left(t,(U^2_t)^\top Y^2_t\right)\right\|_{\mathrm{F}}^2\mathrm{d}t+\mathrm{d}M_t,
			\end{aligned}
		\end{equation}
		where
		\begin{equation*}
			\mathrm{d}M_t:=2\left\langle\Delta Y_t,\left[U^1_tb\!\left(t,(U^1_t)^\top Y^1_t\right)-U^2_tb\!\left(t,(U^2_t)^\top Y^2_t\right)\right]\mathrm{d}W_t\right\rangle
		\end{equation*}
		is a local martingale with $M_0 = 0$.
		
		We first estimate intermediate terms that will be useful in evaluating $\mathbb{E}[D_t^2]$. Since $U^1_t$ and $U^2_t$ have orthonormal rows, one has 
		\begin{equation}\label{eq:X difference uniqueness}
			|(U^1_t)^\top Y^1_t-(U^2_t)^\top Y^2_t|\leq|(U^1_t)^\top(Y^1_t-Y^2_t)|+|(U^1_t-U^2_t)^\top Y^2_t|\leq|\Delta Y_t|+\|\Delta U_t\|_{\mathrm{F}}|Y^2_t|.
		\end{equation}
		From Lemma \ref{lem: DLR Gronwall - lg}, there exists a positive constant $K_{2+\varepsilon}(T)$ such that
		\begin{equation}\label{eq: moment 2eps uniqueness}
			\sup_{t\in[0,T]}\mathbb{E}[|Y^1_t|^{2+\varepsilon}]+\sup_{t\in[0,T]}\mathbb{E}[|Y^2_t|^{2+\varepsilon}]\leq 2 K_{2+\varepsilon}(T).
		\end{equation}
	
		Now, fix $n\geq2$ and define the event $\Omega_{n,t}:=\{|Y^1_t|\vee |Y^2_t|\leq n\}$, where the drift and the diffusion have Lipschitz constant equal to $L_n$. The following estimate will be useful when comparing terms inside the expectation of the vector field of $\Delta U_t$: 
		\begin{equation}\label{eq: local difference a b uniqueness}
			\begin{aligned}
				&\mathbb{E}\!\left[|a(t,(U^1_t)^\top Y^1_t)-a(t,(U^2_t)^\top Y^2_t)|^2\mathbbm{1}_{\Omega_{n,t}}\right]+\mathbb{E}\!\left[\|b(t,(U^1_t)^\top Y^1_t)-b(t,(U^2_t)^\top Y^2_t)\|_{\mathrm{F}}^2\mathbbm{1}_{\Omega_{n,t}}\right]\\
				\leq& L_n^2\mathbb{E}\!\left[|(U^1_t)^\top Y^1_t-(U^2_t)^\top Y^2_t|^2\mathbbm{1}_{\Omega_{n,t}}\right]\\
				\leq&2L_n^2\left(\mathbb{E}[|\Delta Y_t|^2]+\|\Delta U_t\|_{\mathrm{F}}^2\mathbb{E}[|Y^2_t|^2]\right)\leq C_TL_n^2\mathbb{E}[D_t].
			\end{aligned}
		\end{equation}
		for a positive constant $C_T$ independent of $n$. From now on, with a slight abuse of notation, $C_T$ will denote a generic positive constant independent of $n$.
		
		We next give an upper bound on the same expectations restricted to the event $\Omega_{n,t}^c=\{|Y^1_t|\vee |Y^2_t|>n\}$. From Chebyshev's inequality and \eqref{eq: moment 2eps uniqueness}, one has
		\begin{equation}\label{eq: tail probability uniqueness}
			\mathbb{P}(\Omega_{n,t}^c)\leq\mathbb{P}(|Y^1_t|>n)+\mathbb{P}(|Y^2_t|>n)\leq\frac{C_T}{n^{2+\varepsilon}},
		\end{equation}
		and, hence, H\"older's inequality gives, for $i,j\in\{1,2\}$,
		\begin{equation}\label{eq: tail second moment uniqueness}
			\mathbb{E}\!\left[|Y^i_t|^2\mathbbm{1}_{\{|Y^j_t|>n\}}\right]\leq\mathbb{E}[|Y^i_t|^{2+\varepsilon}]^{\frac{2}{2+\varepsilon}}\mathbb{P}(|Y^j_t|>n)^{\frac{\varepsilon}{2+\varepsilon}}\leq\frac{C_T}{n^\varepsilon},
		\end{equation}
		which implies
		\begin{equation}\label{eq: tail growth uniqueness}
			\mathbb{E}\!\left[\left(1+|Y^1_t|^2+|Y^2_t|^2\right)\mathbbm{1}_{\Omega_{n,t}^c}\right]\leq\frac{C_T}{n^\varepsilon}.
		\end{equation}
		By the linear-growth assumption, one has
		\begin{equation*}
			|a(t,(U^1_t)^\top Y^1_t)-a(t,(U^2_t)^\top Y^2_t)|^2+\|b(t,(U^1_t)^\top Y^1_t)-b(t,(U^2_t)^\top Y^2_t)\|_{\mathrm{F}}^2\leq 2C_{\mathrm{lgb}}\left(2+|Y^1_t|^2+|Y^2_t|^2\right),
		\end{equation*}
		and therefore \eqref{eq: local difference a b uniqueness} and \eqref{eq: tail growth uniqueness} imply
		\begin{equation}\label{eq: global difference a b uniqueness}
			\mathbb{E}[|a(t,(U^1_t)^\top Y^1_t)-a(t,(U^2_t)^\top Y^2_t)|^2]+\mathbb{E}[\|b(t,(U^1_t)^\top Y^1_t)-b(t,(U^2_t)^\top Y^2_t)\|_{\mathrm{F}}^2]\leq C_TL_n^2\mathbb{E}[D_t]+\frac{C_T}{n^\varepsilon}.
		\end{equation}
		
		We now estimate the difference of the vector fields in the equation for $U$ in \eqref{eq: Delta U_t}. One has
		\begin{equation*}
			\begin{aligned}
				f(t,U^1_t,Y^1_t)-f(t,U^2_t,Y^2_t)=&\left(C_{Y^1_t}^{-1}-C_{Y^2_t}^{-1}\right)\mathbb{E}\!\left[Y^1_ta(t,(U^1_t)^\top Y^1_t)^\top\right](I_{d\times d}-(U^1_t)^\top U^1_t)\\
				&+C_{Y^2_t}^{-1}\mathbb{E}\!\left[(Y^1_t-Y^2_t)a(t,(U^1_t)^\top Y^1_t)^\top\right](I_{d\times d}-(U^1_t)^\top U^1_t)\\
				&+C_{Y^2_t}^{-1}\mathbb{E}\!\left[Y^2_t\left(a(t,(U^1_t)^\top Y^1_t)-a(t,(U^2_t)^\top Y^2_t)\right)^\top\right](I_{d\times d}-(U^1_t)^\top U^1_t)\\
				&+C_{Y^2_t}^{-1}\mathbb{E}\!\left[Y^2_ta(t,(U^2_t)^\top Y^2_t)^\top\right]\left((U^2_t)^\top U^2_t-(U^1_t)^\top U^1_t\right).
			\end{aligned}
		\end{equation*}
		For the first term, using \eqref{eq: A-B -1}, the uniform bound on the inverse covariance matrices, and Cauchy--Schwarz inequality, one has
		\begin{equation*}
			\|C_{Y^1_t}^{-1}-C_{Y^2_t}^{-1}\|_{\mathrm{F}}\leq 2(4\gamma^2) \|C_{Y^1_t}-C_{Y^2_t}\|_{\mathrm{F}}\leq 16\gamma^2 \sqrt{K_2(T)} \mathbb{E}[|\Delta Y_t|^2]^{\frac12},
		\end{equation*}
		where we used
		\begin{equation*}
			C_{Y^1_t}-C_{Y^2_t}=\mathbb{E}\!\left[(Y^1_t-Y^2_t)(Y^1_t)^\top+Y^2_t(Y^1_t-Y^2_t)^\top\right].
		\end{equation*}
		Moreover, by the linear-growth assumption and the second-moment estimate,
		\begin{equation*}
			\begin{aligned}
			\left\|\mathbb{E}[Y^1_ta(t,(U^1_t)^\top Y^1_t)^\top]\right\|_{\mathrm{F}}+\left\|\mathbb{E}[Y^2_ta(t,(U^2_t)^\top Y^2_t)^\top]\right\|_{\mathrm{F}}\leq& 2 \sqrt{K_2(T)} \sqrt{2 C_{\mathrm{lgb}} \left(1+ K_{2}(T)\right) },\\
		  \left\|\mathbb{E}[(Y^1_t-Y^2_t)a(t,(U^1_t)^\top Y^1_t)^\top]\right\|_{\mathrm{F}}\leq & \sqrt{ C_{\mathrm{lgb}} \left(1+ K_{2}(T)\right)}\mathbb{E}[|\Delta Y_t|^2]^{\frac12}.
						\end{aligned}
		\end{equation*}
		The term requiring more care is the one containing the difference of the drift coefficient inside the expectation. By Cauchy--Schwarz inequality and \eqref{eq: global difference a b uniqueness}, one has
		\begin{equation}\label{eq: drift expectation difference uniqueness}
			\begin{aligned}
				&\left\|\mathbb{E}\!\left[Y^2_t\left(a(t,(U^1_t)^\top Y^1_t)-a(t,(U^2_t)^\top Y^2_t)\right)^\top\right]\right\|_{\mathrm{F}}\\
				&\leq\mathbb{E}[|Y^2_t|^2]^{\frac12}\mathbb{E}[|a(t,(U^1_t)^\top Y^1_t)-a(t,(U^2_t)^\top Y^2_t)|^2]^{\frac12}\\
				&\leq C_TL_n\mathbb{E}[D_t]^{\frac12}+\frac{C_T}{n^{\varepsilon/2}},
			\end{aligned}
		\end{equation}
		for a positive constant $C_T$ independent of $n$. Finally, using
		\begin{equation*}
			\|(U^2_t)^\top U^2_t-(U^1_t)^\top U^1_t\|_{\mathrm{F}}\leq2\|U^1_t-U^2_t\|_{\mathrm{F}},
		\end{equation*}
		which holds by orthogonality of $U^1,U^2$, we obtain
		\begin{equation}\label{eq: f uniqueness final estimate}
			\|f(t,U^1_t,Y^1_t)-f(t,U^2_t,Y^2_t)\|_{\mathrm{F}}\leq C_T(1+L_n)\mathbb{E}[D_t]^{\frac12}+\frac{C_T}{n^{\varepsilon/2}}.
		\end{equation}
		
		We next consider the coefficients in the SDE for $\Delta Y_t$ in \eqref{eq: Delta Y_t}. Denote $A_t:=U^1_ta(t,(U^1_t)^\top Y^1_t)-U^2_ta(t,(U^2_t)^\top Y^2_t)$ and $B_t:=U^1_tb(t,(U^1_t)^\top Y^1_t)-U^2_tb(t,(U^2_t)^\top Y^2_t)$. Using 
		$$A_t=(U^1_t-U^2_t)a(t,(U^1_t)^\top Y^1_t)+U^2_t(a(t,(U^1_t)^\top Y^1_t)-a(t,(U^2_t)^\top Y^2_t)),$$ 
		the linear-growth condition together with \eqref{eq: global difference a b uniqueness} yields
		\begin{equation}\label{eq: A uniqueness estimate}
			\mathbb{E}[|A_t|^2]\leq C_T\|\Delta U_t\|_{\mathrm{F}}^2+C_TL_n^2\mathbb{E}[D_t]+\frac{C_T}{n^\varepsilon}\leq C_T(1+L_n^2)\mathbb{E}[D_t]+\frac{C_T}{n^\varepsilon}.
		\end{equation}
		 Similarly,
		\begin{equation}\label{eq: B uniqueness estimate}
			\mathbb{E}[\|B_t\|_{\mathrm{F}}^2]\leq C_T(1+L_n^2)\mathbb{E}[D_t]+\frac{C_T}{n^\varepsilon}.
		\end{equation}
		
		Finally, we can now return to \eqref{eq: Ito D uniqueness}, with the goal of giving a precise upper bound on $D_t$ in $L^2$ and applying Gronwall's lemma. Since $M_t$ is only a local martingale, let $(\delta_m)_{m\in\mathbb{N}}$ be a localizing sequence of stopping times with values in $[0,T]$ such that $M_{t\wedge\delta_m}$ is a martingale. Considering \eqref{eq: Ito D uniqueness} up to $t\wedge\delta_m$, taking expectations, and using $\mathbb{E}[M_{t\wedge\delta_m}]=0$, one obtains
		\begin{equation}\label{eq: expectation D stopped uniqueness}
			\begin{aligned}
				\mathbb{E}[D_{t\wedge\delta_m}] \leq{}&2\int_0^t\mathbb{E}\!\left[\mathbbm{1}_{\{s\leq\delta_m\}}\|\Delta U_s\|_{\mathrm{F}}\|f(s,U^1_s,Y^1_s)-f(s,U^2_s,Y^2_s)\|_{\mathrm{F}}\right]\mathrm{d}s\\
				&+2\int_0^t\mathbb{E}\!\left[\mathbbm{1}_{\{s\leq\delta_m\}}|\Delta Y_s||A_s|\right]\mathrm{d}s+\int_0^t\mathbb{E}\!\left[\mathbbm{1}_{\{s\leq\delta_m\}}\|B_s\|_{\mathrm{F}}^2\right]\mathrm{d}s.
			\end{aligned}
		\end{equation}
		Since $U^1_t,U^2_t$ are deterministic, $\|\Delta U_t\|_{\mathrm{F}}\leq\mathbb{E}[D_t]^{1/2}$. Hence, from \eqref{eq: f uniqueness final estimate} and Young's inequality,
		\begin{equation*}
			2\|\Delta U_t\|_{\mathrm{F}}\|f(t,U^1_t,Y^1_t)-f(t,U^2_t,Y^2_t)\|_{\mathrm{F}}\leq C_T(1+L_n^2)\mathbb{E}[D_t]+\frac{C_T}{n^\varepsilon}.
		\end{equation*}
		Moreover, by Cauchy--Schwarz inequality, Young's inequality and \eqref{eq: A uniqueness estimate},
		\begin{equation*}
			2\mathbb{E}[|\Delta Y_t||A_t|]\leq\mathbb{E}[|\Delta Y_t|^2]+\mathbb{E}[|A_t|^2]\leq C_T(1+L_n^2)\mathbb{E}[D_t]+\frac{C_T}{n^\varepsilon}.
		\end{equation*}
		Together with \eqref{eq: B uniqueness estimate}, this gives
		\begin{equation*}
			\mathbb{E}[D_{t\wedge\rho_m}]\leq C_T(1+L_n^2)\int_0^t\mathbb{E}[D_s]\,\mathrm{d}s+\frac{C_T}{n^\varepsilon}.
		\end{equation*}
		Letting $m\to\infty$ and using Fatou's lemma, we obtain, for every $n\geq2$,
		\begin{equation}\label{eq: master uniqueness D}
			\mathbb{E}[D_t]\leq C_T(1+L_n^2)\int_0^t\mathbb{E}[D_s]\,\mathrm{d}s+\frac{C_T}{n^\varepsilon},\qquad t\in[0,T].
		\end{equation}
	
		Via applying Gronwall's lemma to \eqref{eq: master uniqueness D}, one has that 
		\begin{equation*}
			\begin{aligned}
		 \mathbb{E}[D_t] \leq &\frac{C_T}{n^\varepsilon} \mathrm{exp}\left( C_T(1+ L_n^2) t \right) \\
		 \leq & \frac{C_T}{n^\varepsilon} \mathrm{exp}\left( C_T(1+ C_{\mathrm{loc}}^2 \log n) t \right) = C_T e^{C_T t} n^{ C_T C_{\mathrm{loc}}^2 t - \varepsilon}.
		 \end{aligned}
		\end{equation*}
		Therefore, for $t < \frac{\varepsilon}{C_T C_{\mathrm{loc}}^2}$ letting $n \to + \infty$ we have $\mathbb{E}[D_t]=0$ for all $t \in [0, \frac{\varepsilon}{C_T C_{\mathrm{loc}}^2})$. Iterating the same argument over the time $t$, one concludes that $\mathbb{E}[D_t]=0$ for all $t \in [0, T]$.
		
		Hence $D_t=0$ almost surely for every fixed $t\in[0,T]$, and therefore $U^1_t=U^2_t$ and $Y^1_t=Y^2_t$ almost surely for every fixed $t\in[0,T]$. Using the a.s.\ continuity of both strong DO solutions, we conclude that the two solutions are indistinguishable for all $t\in[0,T]$, and therefore the strong DO solution is pathwise unique.
	\end{proof}
\end{Proposition}

\section*{Conclusion}
\addcontentsline{toc}{section}{Conclusion}
In this article, we established the existence of solutions of the DO equations under local-Lipschitzianity combined with linear-growth bound of the drift and the diffusion. 
We also proved uniqueness under sub-logarithmic growth of the local Lipschitz constant.

Furthermore, it is also interesting to expand the well-posed DLRA framework for SDEs to other conditions, such as weak-monotonicity of the drift, which is a work in progress.

\section*{Acknowledgements}
This work has been supported by the Swiss National Science Foundation under the
Project n. 200518 “Dynamical low rank methods for uncertainty quantification and data assimilation”.

\printbibliography

\end{document}